\documentclass{amsart}
\usepackage{amsmath,amsthm,amsfonts,amssymb}

\numberwithin{equation}{section}
\theoremstyle{plain}

\newtheorem{theorem}{Theorem}[section]

\newtheorem{definition}[theorem]{Definition}

\newtheorem{lemma}[theorem]{Lemma}

\newtheorem{remark}[theorem]{Remark}

\newcommand{\one}{{{\rm 1\mkern-1.5mu}\!{\rm I}}}
\newcommand{\eps}{{\epsilon}}

\begin{document}

\title[Inequality of quenched and averaged rate functions]{Differing averaged
and quenched large deviations for Random Walks in Random
Environments in dimensions two and three}

\author{Atilla Yilmaz}
\address{Faculty of Mathematics\\
Weizmann Institute of Science}
\curraddr{Department of Mathematics\\
University of California, Berkeley}
\email{atilla@math.berkeley.edu}
\author{Ofer Zeitouni}
\address{Faculty of Mathematics,
Weizmann Institute of Science {\tt and} 
School of Mathematics, University of Minnesota}
\email{zeitouni@math.umn.edu}
\date{September 28, 2009. Revised March 15, 2010}

%\thanks{This research was supported partially by a grant from the Israeli Science Foundation.}

\begin{abstract}
We consider the quenched and the averaged (or annealed)
large deviation rate functions 
$I_q$ and $I_a$ for space-time and (the usual) 
space-only RWRE on $\mathbb{Z}^d$. By Jensen's inequality, $I_a\leq I_q$.

In the space-time case, when $d\geq3+1$, $I_q$ and $I_a$ are known to 
be equal on an open set containing the typical velocity $\xi_o$. When $d=1+1$, we prove that $I_q$ and $I_a$ are equal only at $\xi_o$. Similarly, when $d=2+1$, we show that $I_a<I_q$ on a punctured neighborhood of $\xi_o$.

In the space-only case, we provide a class of non-nestling walks on $\mathbb{Z}^d$ with $d=2$ or $3$, and prove that $I_q$ and $I_a$ are not identically equal on any open set containing $\xi_o$ whenever the walk is in that class. This is very different from the known results for non-nestling walks on $\mathbb{Z}^d$ with $d\geq4$.
\end{abstract}

\maketitle

\section{Introduction}

\subsection{The models}\label{SubsectionModels} Consider a discrete time Markov chain on the $d$-dimensional integer lattice $\mathbb{Z}^d$ with $d\geq1$. For any $x,z\in\mathbb{Z}^d$, denote the transition probability from $x$ to $x+z$ by $\pi(x,x+z)$. Refer to the transition vector $\omega_x:=(\pi(x,x+z))_{z\in\mathbb{Z}^d}$ as the \textit{environment} at $x$. If the environment $\omega:=(\omega_x)_{x\in\mathbb{Z}^d}$ is sampled from a probability space $(\Omega,\mathcal{B},\mathbb{P})$, then this process is called \textit{random walk in a random environment} (RWRE). Here, $\mathcal{B}$ is the Borel $\sigma$-algebra corresponding to the product topology.

For every $y\in\mathbb{Z}^d$, define the shift $T_y$ on $\Omega$ by $\left(T_y\omega\right)_x:=\omega_{x+y}$. In order to have some statistical homogeneity in the environment, $\mathbb{P}$ is generally assumed to be stationary and ergodic with respect to $(T_y)_{y\in\mathbb{Z}^d}$. In this paper, we will make the stronger assumption that
\begin{equation}\label{cakmaeminabi}
\mbox{$\mathbb{P}$ is a product measure with equal marginals.}
\end{equation}
In other words, $\omega=(\omega_x)_{x\in\mathbb{Z}^d}$ is a collection of independent and identically distributed (i.i.d.) random vectors.

%OO1
The set $\mathcal{R}:=\{z\in\mathbb{Z}^d: \mathbb{P}(\pi(0,z)>0)>0\}$ 
is the \textit{range} of allowed steps of the walk (here and throughout,
we often 
use $0$ to denote the origin in $\mathbb{Z}^d$ when no confusion
occurs). 
Let $(e_i)_{i=1}^d$ denote the canonical basis for $\mathbb{Z}^d$. The walk is said to be \textit{space-time} if
\begin{equation}\label{space-time}
\mathcal{R}=\mathcal{R}_{st}:=\{(z_1,\ldots,z_d)\in\mathbb{Z}^d: |z_1|+\cdots+|z_{d-1}|=1, z_d=1\},
\end{equation}
and it is said to be \textit{space-only} if
\begin{equation}\label{space-only}
\mathcal{R}=\mathcal{R}_{so}:=\{\pm e_i\}_{i=1}^d.
\end{equation}
In either case,
we will assume throughout the paper
that there exists a $\kappa>0$ such 
that $\mathbb{P}(\pi(0,z)\geq\kappa)=1$ for every $z\in\mathcal{R}$. 
This condition is known as \textit{uniform ellipticity}.

Space-time is a natural term for the case
(\ref{space-time}) since then,
  the walk decomposes into two parts. Its projection on the $e_d$-axis is deterministic and can be identified with time. The motion in the span of $(e_i)_{i=1}^{d-1}$ can be thought of as a variation of space-only RWRE where the environment is freshly sampled at each time step. To emphasize this decomposition, we will write the dimension as $d=(d-1)+1$. For example, when $d=3$, we will say that the dimension is $2+1$.

For every $x\in\mathbb{Z}^d$ and $\omega\in\Omega$, the Markov chain with environment $\omega$ induces a probability measure $P_x^\omega$ on the space of paths starting at $x$. Statements about $P_x^\omega$ that hold for $\mathbb{P}$-a.e.\ $\omega$ are referred to as \textit{quenched}. Statements about the semi-direct product $P_x:=\mathbb{P}\times P_x^\omega$ are referred to as 
\textit{averaged} (or \textit{annealed}). 
Expectations under $\mathbb{P}, P_x^\omega$ and $P_x$ 
are denoted by $\mathbb{E}, E_x^\omega$ and $E_x$, respectively. 

See \cite{ZeitouniSurvey06} for a survey of results and open problems on RWRE.

It is clear that no model satisfies both (\ref{space-time}) and (\ref{space-only}). Nevertheless, it turns out that many of the results that hold for space-only RWRE are valid under also the space-time assumption, and it is fair to say that space-time RWRE is easier to analyze than space-only RWRE because (\ref{space-time}) ensures that the walk never visits the same point more than once.

\subsection{Regeneration times}\label{regss}

In the next subsection, we will give a brief survey of the previous results on large deviations for RWRE in order to put the present work in context. Some of these results involve certain random times which are introduced below for convenience.

Let $\left(X_n\right)_{n\geq0}$ denote the path of a space-only RWRE. Consider a unit vector $\hat{u}\in\mathcal{S}^{d-1}$.
Define a sequence $\left(\tau_m\right)_{m\geq0}$ of random times, which are referred to as \textit{regeneration times} (relative to $\hat{u}$), by $\tau_o:=0$ and
$$\tau_{m}:=\inf\left\{j>\tau_{m-1}:\langle X_i,\hat{u}\rangle<\langle X_j,\hat{u}\rangle\leq\langle X_k,\hat{u}\rangle\mbox{ for all }i,k\mbox{ with }i<j<k\right\}$$
for every $m\geq1$. 
(Regeneration times first appeared in the work of Kesten \cite{KestenReg} on one-dimensional RWRE. They were adapted to the multidimensional setting
by Sznitman and Zerner, c.f.\ \cite{SznitmanZerner99}.)
Because we assumed the environment 
$\omega=(\omega_x)_{x\in\mathbb{Z}^d}$ to be
an i.i.d.\ collection,
if the walk is directionally transient relative to $\hat{u}$, i.e., if $P_o\left(\lim_{n\to\infty}\langle X_n,\hat{u}\rangle=\infty\right)=1$,
then $P_o\left(\tau_m<\infty\right)=1$ for every $m\geq1$.
In this setup,
as noted in \cite{SznitmanZerner99}, the significance of $\left(\tau_m\right)_{m\geq1}$ is due to the fact that 
$$\left(X_{\tau_m+1}-X_{\tau_m},X_{\tau_m+2}-X_{\tau_m},\ldots,X_{\tau_{m+1}}-
X_{\tau_m}, \tau_{m+1}-\tau_m\right)_{m\geq1}$$ 
is an i.i.d.\ sequence under $P_o$. 

The walk is said to satisfy Sznitman's transience condition (\textbf{T}) if
$$E_o\left[\sup_{1\leq i\leq\tau_1}
\exp\left\{c_1\left|X_i\right|\right\}\right]<\infty\mbox{ for some }c_1>0.$$
%OO Is that indeed what we meant? Or do we want $\ell_2$? Both are OK but 
%OO we used sometimes estimates in boxes, which is more natural in l_\infty
%AA I think we want $\ell_2$. At least, that's what I always had in mind. All the constants, etc. have been chosen accordingly.
(Here and throughout, the norm $|\cdot|$ denotes the
$\ell_2$ norm).
When $d\geq2$, Sznitman \cite{SznitmanT} proves that (\ref{cakmaeminabi}), (\ref{space-only}) and (\textbf{T}) imply a \textit{ballistic} law of large numbers (LLN), an averaged central limit theorem and certain large deviation estimates.

Condition (\textbf{T}) holds as soon as the walk is \textit{non-nestling} relative to $\hat{u}$, i.e., when the random drift vector
\begin{equation}\label{nonnestabi}
v(\omega):=\sum_{z\in\mathcal{R}}\pi(0,z)z\qquad\mbox{satisfies}\qquad\mathrm{ess}\inf_{\mathbb{P}}\langle v(\cdot),\hat{u}\rangle>0.
\end{equation}
The walk is said to be non-nestling if it is non-nestling relative to some unit vector. Otherwise, it is referred to as \textit{nestling}. In the latter case, the convex hull of the support of the law of $v(\cdot)$ contains the origin.

In the case of space-time RWRE, regeneration times are defined naturally by
 taking $\hat{u}=e_d$ and $\tau_m=m$ for every $m\geq1$. 
Clearly, the space-time walk is always non-nestling relative to $\hat u=e_d$.

\subsection{Previous results on large deviations for RWRE}

Recall that a sequence $\left(Q_n\right)_{n\geq1}$ of probability measures on a topological space $\mathbb{X}$ is said to satisfy the \textit{large deviation principle} (LDP) with a 
rate function $I:\mathbb{X}\to\mathbb{R}_+\cup\{\infty\}$ if $I$ is  
lower semicontinuous and for any measurable set $G$, $$-\inf_{x\in G^o}I(x)\leq\liminf_{n\to\infty}\frac{1}{n}\log Q_n(G)\leq\limsup_{n\to\infty}\frac{1}{n}\log Q_n(G)\leq-\inf_{x\in\bar{G}}I(x).$$ Here, $G^o$ is the interior of $G$, and $\bar{G}$ its closure. See \cite{DemboZeitouniBook} for general background regarding large deviations.

We will focus on the following large deviation principles
%OO
%AA In the paragraph below, we are indicating the conditions (i.i.d., mixing, ergodic, etc.) under which these results have been proved.
for walks in uniformly elliptic environments.
\begin{theorem}[Quenched LDP]\label{qLDPgen}
For $\mathbb{P}$-a.e.\ $\omega$, $\left(P_o^\omega\left(\frac{X_n}{n}\in\cdot\,\right)\right)_{n\geq1}$ satisfies the LDP with a deterministic and convex rate function $I_q$.
\end{theorem}

\begin{theorem}[Averaged LDP]\label{aLDPgen}
$\left(P_o\left(\frac{X_n}{n}\in\cdot\,\right)\right)_{n\geq1}$ satisfies the LDP with a convex rate function $I_a$.
\end{theorem}

There are many works on large deviations for space-only RWRE. We briefly mention them in chronological order. Greven and den Hollander \cite{GdH94} prove Theorem \ref{qLDPgen} for walks on $\mathbb{Z}$ under the i.i.d.\ environment assumption. They provide a formula for $I_q$ and show that its graph typically has flat pieces. Zerner \cite{Zerner98} establishes Theorem \ref{qLDPgen} for nestling walks on $\mathbb{Z}^d$ in i.i.d.\ environments. Comets, Gantert and Zeitouni \cite{CGZ00} generalize the result of \cite{GdH94} to walks on $\mathbb{Z}$ in stationary and ergodic environments. Also, they prove Theorem \ref{aLDPgen} for walks on $\mathbb{Z}$ in i.i.d.\ environments and give a formula that links $I_a$ to $I_q$. Varadhan \cite{Varadhan03} generalizes Zerner's result to stationary and ergodic environments without any nestling assumption. He also proves Theorem \ref{aLDPgen} for walks on $\mathbb{Z}^d$ in i.i.d.\ environments and gives a variational formula for $I_a$. Rassoul-Agha \cite{FirasLDP04} generalizes the latter result of \cite{Varadhan03} to certain mixing environments. Rosenbluth \cite{RosenbluthThesis} gives an alternative proof of Theorem \ref{qLDPgen} for walks on $\mathbb{Z}^d$ in stationary and ergodic environments, and provides a variational formula for $I_q$. Yilmaz \cite{YilmazQuenched} generalizes the result of \cite{RosenbluthThesis} to a so-called level-2 LDP. Berger \cite{BergerZeroSet}, Peterson and Zeitouni \cite{JonOferLDP08}, and Yilmaz \cite{YilmazAveraged} obtain certain qualitative properties of $I_a$. Rassoul-Agha and Sepp\"al\"ainen \cite{RAS} generalize the result of \cite{RosenbluthThesis} to a so-called level-3 LDP.

In the case of space-time RWRE, Rassoul-Agha and Sepp\"al\"ainen \cite{FirasTimoSpaceTime} prove Theorem \ref{qLDPgen} by adapting the quenched argument in \cite{Varadhan03}. Theorem \ref{aLDPgen} does not require any work. 
Indeed,
Assumption (\ref{space-time}) implies that the walk under $P_o$ is a sum of i.i.d.\ increments. The common distribution of these increments is $(q(z))_{z\in\mathcal{R}}$ where $q(z):=\mathbb{E}[\pi(0,z)]$ for every $z\in\mathcal{R}$. 
Therefore, Theorem \ref{aLDPgen} 
in the space-time setup
is simply Cram\'er's theorem, c.f.\ \cite{DemboZeitouniBook}.

In addition to the works mentioned in the last two paragraphs, there are two more results on large deviations for RWRE that are relevant to this paper. 
We state them in detail.

\begin{theorem}[Yilmaz \cite{YilmazSpaceTime}]\label{QequalsAst}
Assume (\ref{cakmaeminabi}) and (\ref{space-time}). If $d\geq3+1$, then $I_q=I_a$ on a set $\mathcal{A}_{st}\times\{e_d\}$ containing the LLN velocity $\xi_o$, where $\mathcal{A}_{st}$ is an open subset of $\mathbb{R}^{d-1}$.
\end{theorem}

\begin{theorem}[Yilmaz \cite{YilmazQequalsA}]\label{QequalsAnn}
Assume (\ref{cakmaeminabi}), (\ref{space-only}), $d\geq4$, and that Sznitman's (\textbf{T}) condition holds for some $\hat{u}\in\mathcal{S}^{d-1}$.
\begin{itemize}
\item[(a)] If the walk is non-nestling, then $I_q=I_a$ on an open set $\mathcal{A}_{so}$ containing the LLN velocity $\xi_o$.
\item[(b)] If the walk is nestling, then
\begin{itemize}
\item[(i)] $I_q=I_a$ on an open set $\mathcal{A}_{so}^+$,
\item[(ii)] there exists a $(d-1)$-dimensional smooth surface patch $\mathcal{A}_{so}^b$ such that $\xi_o\in\mathcal{A}_{so}^b\subset\partial\mathcal{A}_{so}^+$,
\item[(iii)] the unit vector $\eta_o$ normal to $\mathcal{A}_{so}^b$ (and pointing inside $\mathcal{A}_{so}^+$) at $\xi_o$ satisfies $\langle\eta_o,\xi_o\rangle>0$, and
\item[(iv)] $I_q(t\xi)=tI_q(\xi)=tI_a(\xi)=I_a(t\xi)$ for every $\xi\in\mathcal{A}_{so}^b$ and $t\in[0,1]$.
\end{itemize}
\end{itemize}
\end{theorem}
It is worthwhile to emphasize that the equality $I_q=I_a$ does not extend, in the setup of Theorems \ref{QequalsAst} and \ref{QequalsAnn}, to the whole space. Indeed, for any $d\geq1$,
\begin{equation}\label{husnukon}
\mbox{$I_a<I_q$ at the extremal points of the domain of $I_a$.}
\end{equation}
By continuity, this inequality holds also at some interior points. See Proposition 4 of \cite{YilmazQequalsA} for details.

\subsection{Our results}

For space-time RWRE, it is natural to ask whether Theorem \ref{QequalsAst} can be generalized to $d\geq1+1$ or $2+1$. The answer turns out to be no.
\begin{theorem}\label{QneqAst}
Assume (\ref{cakmaeminabi}) and (\ref{space-time}). If $d=1+1$, then $I_q(\xi)=I_a(\xi)<\infty$ if and only if $\xi=\xi_o$, the LLN velocity.
\end{theorem}
\begin{theorem}\label{QneqAstyeni}
Assume (\ref{cakmaeminabi}) and (\ref{space-time}). If $d=2+1$, then $I_a<I_q$ on a 
set $(\mathcal{G}_{st}\times\{e_3\})\setminus\{\xi_o\}$,
 where $\mathcal{G}_{st}\subset\mathbb{R}^2$ is open 
and $\mathcal{G}_{st}\times\{e_3\}$ contains $\xi_o$.
\end{theorem}

In the case of space-only RWRE on $\mathbb{Z}$, 
a consequence of Comets et al.\ \cite{CGZ00}, Proposition 5, is
 that $I_q(\xi)=I_a(\xi)<\infty$ if and only if $\xi=0$ or $I_a(\xi)=0$. In particular, Theorem \ref{QequalsAnn} 
cannot be generalized to $d\geq1$.
Our next result shows that the conclusion of Theorem \ref{QequalsAnn} is false
for a class of space-only RWRE's in dimensions $d=2,3$. 
\begin{definition}\label{gaffur}
Assume $d\geq2$, and fix a triple $p=(p^+,p^o,p^-)$ of positive real numbers such that $p^-<p^+$ and $p^++p^o+p^-=1$. For any $\epsilon>0$, a probability measure $\mathbb{P}$ on $(\Omega,\mathcal{B})$ is said to be in class $\mathcal{M}_\epsilon(d,p)$ if
\begin{itemize}
\item[(a)] (\ref{cakmaeminabi}) and (\ref{space-only}) hold,
\item[(b)] $\mathbb{P}(\pi(0,e_d)=p^+$, $\pi(0,-e_d)=p^-)=1$,
\item[(c)] $\mathbb{P}(\epsilon/2<|\pi(0,e_1)-\frac{p^o}{2(d-1)}|<\epsilon)=1$, and
\item[(d)] $\mathbb{P}$ is invariant under the rotations of $\mathbb{Z}^d$ that preserve $e_d$. (We will refer to this as \textit{isotropy}.)
\end{itemize}
\end{definition}

\begin{theorem}\label{QneqAnn}
Assume $d=2$ or $3$. Fix a triple $p=(p^+,p^o,p^-)$ as in Definition \ref{gaffur}. 
Then there exists an $\epsilon_o=\epsilon_o(p)$
such that if $\epsilon<\epsilon_o$ and
$\mathbb{P}$ is in class $\mathcal{M}_\epsilon(d,p)$, then
the quenched and the averaged rate functions $I_q$ and $I_a$ are not identically equal on any open set containing the LLN velocity $\xi_o$.
\end{theorem}

The proofs of our results are based on a technique that combines the so-called \textit{fractional moment} method with a certain change of measure (which we will refer to as \textit{tilting the environment}). This technique has been developed for analyzing the so-called polymer pinning model, c.f.\ \cite{DeGiLaTo, ToniCrit, GiLaTo}, and it has been recently 
refined 
by Lacoin \cite{Lacoin} for obtaining certain lower bounds 
for the free energy of directed polymers in random environments.
Comparing with the polymer setup, an extra complication occurs
in the RWRE model due to the dependence
of the transition probabilities of the walk on the environment. 
(In the polymer model discussed above, the walk is a simple random walk,
and the environment only appears in the evaluation of exponential moments
with respect to the random walk.)
The difficulty in the RWRE setup, and much of our work, 
lies in overcoming this dependency.
For space-time RWRE, this task is greatly simplified because
each site is visited at most once. For space-only RWRE, where this is not true,
we employ a perturbative approach that unfortunately restricts the class
of models considered, see Section \ref{sec-openprob} for
further comments.

Here is how the rest of the paper is organized: In Section \ref{SpaceTimeSection}, we consider space-time RWRE and prove Theorems \ref{QneqAst} and \ref{QneqAstyeni} by adapting the relevant arguments given in \cite{Lacoin}. In Section \ref{SpaceOnlySection}, we focus on space-only walks that are non-nestling relative to $e_d$, and modify the previous proofs by making use of regeneration times. This way, we establish a result (see Theorem \ref{nesin}) analogous to Theorems \ref{QneqAst} and \ref{QneqAstyeni}. The only difference is that Theorem \ref{nesin} is valid under a certain correlation condition, c.f.\ (\ref{theassumption}). Finally, we prove Theorem \ref{QneqAnn} by checking that (\ref{theassumption}) holds whenever $\mathbb{P}$ is in class $\mathcal{M}_\epsilon(d,p)$ with some triple $p$ (as in Definition \ref{gaffur}) and a sufficiently small $\epsilon>0$.

\section{Inequality of the rate functions for space-time RWRE}\label{SpaceTimeSection}

\subsection{Reducing to a fractional moment estimate}

Assume $d\geq1+1$. Recall (\ref{space-time}). Consider a space-time random walk on $\mathbb{Z}^d$ in a uniformly elliptic and i.i.d.\ environment. For every $\theta\in\mathbb{R}^d$, define $$\phi(\theta):=\sum_{z\in\mathcal{R}}\mathrm{e}^{\langle\theta,z\rangle} q(z)$$ where $q(z):=\mathbb{E}[\pi(0,z)]$. Since the walk visits every point at most once, $E_o\left[\exp\{\langle\theta,X_N\rangle\}\right]=\phi(\theta)^N$ for every $N\geq1$.

Define the logarithmic moment generating functions $$\Lambda_q(\theta):=\lim_{N\to\infty}\frac{1}{N}\log E_o^\omega\left[\exp\{\langle\theta,X_N\rangle\}\right]\quad\mbox{and}\quad\Lambda_a(\theta):=\lim_{N\to\infty}\frac{1}{N}\log E_o\left[\exp\{\langle\theta,X_N\rangle\}\right]=\log\phi(\theta).$$
By Varadhan's Lemma, c.f.\ \cite{DemboZeitouniBook}, $\Lambda_q(\theta)=\sup_{\xi\in\mathbb{R}^d}\left\{\langle\theta,\xi\rangle - I_q(\xi)\right\}=I_q^*(\theta)$, the convex conjugate of $I_q$ at $\theta$. Similarly, $\Lambda_a(\theta)=\log\phi(\theta)=I_a^*(\theta)$.

For every $N\geq1$, $\theta\in\mathbb{R}^d$ and $\omega\in\Omega$, define $$W_N(\theta,\omega) := E_o^\omega[\exp\{\langle\theta,X_N\rangle-N\log\phi(\theta)\}].$$
Given any $\alpha\in(0,1)$, Jensen's inequality and the bounded convergence theorem imply that
\begin{align}
\Lambda_q(\theta)-\log\phi(\theta)&=\lim_{N\to\infty}\frac{1}{N}\log W_N(\theta,\cdot)=\mathbb{E}\left[\lim_{N\to\infty}\frac{1}{N}\log W_N(\theta,\cdot)\right]\nonumber\\
&=\lim_{N\to\infty}\frac{1}{N}\mathbb{E}\left[\log W_N(\theta,\cdot)\right]=\lim_{N\to\infty}\frac{1}{N\alpha}\mathbb{E}\left[\log W_N(\theta,\cdot)^\alpha\right]\nonumber\\
&\leq\limsup_{N\to\infty}\frac{1}{N\alpha}\log\mathbb{E}\left[W_N(\theta,\cdot)^\alpha\right]\label{fracmom}\\
&\leq\lim_{N\to\infty}\frac{1}{N\alpha}\log\left(\mathbb{E}\left[W_N(\theta,\cdot)\right]\right)^\alpha=0.\nonumber
\end{align}
\begin{lemma}\label{fracmomlemma}
Assume (\ref{cakmaeminabi}) and (\ref{space-time}). Fix any $\alpha\in(0,1)$. If $d=1+1$, then
\begin{equation}\label{dukkan}
\limsup_{N\to\infty}\frac{1}{N}\log\mathbb{E}
\left[W_N(\theta,\cdot)^\alpha\right]<0
\end{equation}
whenever $\theta\notin sp\{e_2\}$, the one-dimensional vector space spanned by $e_2$.
\end{lemma}
\begin{lemma}\label{fracmomlemmayeni}
Assume (\ref{cakmaeminabi}) and (\ref{space-time}). Fix any $\alpha\in(0,1)$. If $d=2+1$, then there exists a $\beta>0$ such that (\ref{dukkan}) holds whenever $\mathrm{dist}(\theta, sp\{e_3\})\in(0,\beta)$.
\end{lemma}
\begin{remark}
For every $\theta\in sp\{e_d\}$, (\ref{space-time}) implies that $W_N(\theta,\cdot)=1$ and $\Lambda_q(\theta)=\log\phi(\theta)$.
\end{remark}
When $d=1+1$, it follows from (\ref{fracmom}) and Lemma \ref{fracmomlemma} that $\Lambda_q(\cdot)<\log\phi(\cdot)$ on $\{\theta\in\mathbb{R}^2: \theta\notin sp\{e_2\}\}$. By convex duality, $I_a<I_q$ on $\{\nabla\log\phi(\theta):\theta\notin sp\{e_2\}\}$. It is easy to see that the latter set is equal to $((-1,1)\times\{e_2\})\setminus\{\xi_o\}$. In combination with (\ref{husnukon}), this proves Theorem \ref{QneqAst}.

Similarly, when $d=2+1$, Lemma \ref{fracmomlemmayeni} implies that $I_a<I_q$ on $\{\nabla\log\phi(\theta): \mathrm{dist}(\theta,sp\{e_3\})\in(0,\beta)\}$. One can check that this set is of the form $(\mathcal{G}_{st}\times\{e_3\})\setminus\{\xi_o\}$ where $\mathcal{G}_{st}\subset\mathbb{R}^2$ is open and $\mathcal{G}_{st}\times\{e_3\}$ contains $\xi_o$. This proves Theorem \ref{QneqAstyeni}.

The rest of this section is devoted to proving Lemmas \ref{fracmomlemma} and \ref{fracmomlemmayeni}.

\subsection{Decomposing into paths}\label{decomss}

Assume $d=1+1$ or $2+1$. Let $\mathbb{V}_d:=\mathbb{Z}^{d-1}\times\{0\}\subset\mathbb{Z}^d$. Fix an $n$ of the form $k^2$, with $k$ an integer
to be determined later 
(e.g., for $d=1+1$, this $n$ is chosen  so that the conclusion of
Lemma
\ref{lessthanhalf} below holds).
When $d=1+1$, let
\begin{equation}\label{azkalbir}
J_y:=\left[(y'-\frac{1}{2})\sqrt{n},(y'+\frac{1}{2})\sqrt{n}\right)\times\{0\}\subset\mathbb{R}^2
\end{equation} for every $y=(y',0)\in\mathbb{V}_2$. Similarly, when $d=2+1$, let $$J_y:=\left[(y'-\frac{1}{2})\sqrt{n},(y'+\frac{1}{2})\sqrt{n}\right)\times\left[(y''-\frac{1}{2})\sqrt{n},(y''+\frac{1}{2})\sqrt{n}\right)\times\{0\}\subset\mathbb{R}^3$$ for every $y=(y',y'',0)\in\mathbb{V}_3$.

Take $N=nm$ for some $m\geq1$. For every $\theta\in\mathbb{R}^d$, $\omega\in\Omega$ and $Y=(y_1,\ldots,y_m)\in(\mathbb{V}_d)^m$, define
\begin{equation}\label{azkaliki}
\bar{W}_N(\theta,\omega,Y):=E_o^\omega[\exp\{\langle\theta,X_N\rangle-N\log\phi(\theta)\}, X_{jn}-\lfloor jn\xi(\theta)\rfloor\in J_{y_j}\mbox{ for every }j\leq m]
\end{equation} where $\xi(\theta)=\nabla\log\phi(\theta)$. (For $u\in\mathbb{R}^d$, $\lfloor u\rfloor$ denotes the closest element of $\mathbb{Z}^d$ to $u$. If there is more than one closest element, then take the one whose index is the smallest with respect to the lexicographic order.) Note that $\langle\xi(\theta),e_d\rangle=1$ because $\langle z,e_d\rangle=1$ for every $z\in\mathcal{R}_{st}$.

Since $\mathbb{V}_d$ is contained in the disjoint union $\cup_{y\in\mathbb{V}_d}J_y$, we see that $W_N(\theta,\omega)=\sum_{Y}\bar{W}_N(\theta,\omega,Y)$. Hence, $W_N(\theta,\omega)^\alpha\leq\sum_{Y}\bar{W}_N(\theta,\omega,Y)^\alpha$ by subadditivity, and 
\begin{equation}\label{asama}
\mathbb{E}[W_N(\theta,\cdot)^\alpha]\leq\sum_{Y}\mathbb{E}\left[\bar{W}_N(\theta,\cdot,Y)^\alpha\right].
\end{equation}

In the rest of this section, we will treat the cases $d=1+1$ and $d=2+1$ separately.

\subsection{Tilting along a path ($d=1+1$)}\label{tiltsubsec}

Our aim is to prove Lemma \ref{fracmomlemma} which states that $\mathbb{E}[W_N(\theta,\cdot)^\alpha]$ decays exponentially in $N$. Let us say a few words about our strategy. For any function $g(\theta,\cdot)$ on $\Omega$,
\begin{align}
\mathbb{E}[W_N(\theta,\cdot)^\alpha]&=\mathbb{E}\left[(W_N(\theta,\cdot)g(\theta,\cdot))^\alpha g(\theta,\cdot)^{-\alpha}\right]\nonumber\\
&\leq\mathbb{E}\left[W_N(\theta,\cdot)g(\theta,\cdot)\right]^\alpha\mathbb{E}\left[g(\theta,\cdot)^{-\frac{\alpha}{1-\alpha}}\right]^{1-\alpha}\label{firstsecondtermbh}
\end{align}
by H\"older's inequality. For every $i\geq1$, $E_{X_i}^\omega[\exp\{\langle\theta,X_{i+1}-X_i\rangle-\log\phi(\theta)\}]$ and $\langle\theta,v(T_{X_i}\omega)-\xi_o\rangle$ are correlated, c.f.\ (\ref{FKG}), where $v(\cdot)$ denotes the random drift vector. We could try to exploit this fact by tilting the environment at the points on the path in a clever way, e.g., by choosing a $g(\theta,\cdot)$ that penalizes the environments for which $\frac1{N}\sum_{i=1}^N\langle\theta,v(T_{X_i}\omega)-\xi_o\rangle$ deviates from zero. This way, we could make the first expectation in (\ref{firstsecondtermbh}) small. However, there is a problem: we do not know where the path is, and if we naively tilt the environment everywhere, then the second expectation in (\ref{firstsecondtermbh}) might become too large. Fortunately, it is possible to resolve this issue by first decomposing $\mathbb{E}[W_N(\theta,\cdot)^\alpha]$ as in (\ref{asama}) (so that we know roughly where the path is), and then tilting the environment on a tube which contains most of the path with a high probability.

Given $m\geq1$, $\theta\notin sp\{e_2\}$, $C_1\geq1$ and $Y=(y_1,\ldots,y_m)\in(\mathbb{V}_2)^m$,
let 
\begin{equation}\label{bicey}
B_j:=\{(s,i)\in\mathbb{Z}^2: (j-1)n\leq i<jn, \left|(s,i)-\lfloor i\xi(\theta)\rfloor-\sqrt{n}y_{j-1}\right|\leq C_1\sqrt{n}\}
\end{equation} for every $j\in\{1,\ldots,m\}$. Here, $y_o=(0,0)$. 
Recall that $n=k^2$ for some integer $k$.

Fix a large $K$ and a small $\delta_n$, both to be determined later
(depending on the choice of $\alpha$, see \eqref{star1}, \eqref{star2}
and Lemma \ref{lessthanhalf}). 
Define $f_K(u):=-K\one_{u\geq\mathrm{e}^{K^2}}$ and
\begin{equation}\label{cii}
g(\theta,\omega,Y):=
\exp\sum_{j=1}^mf_K\left(\delta_nD(B_j)\right)
>0\,,
\end{equation} where
\begin{equation}\label{abbrev}
D(B_j):=\sum_{(s,i)\in B_j}a(\theta,(s,i))\qquad\mbox{for every $j\in\{1,\ldots,m\}$,}
\end{equation} and $a(\theta,x):=\langle\theta,v(T_x\omega)-\xi_o\rangle$ for every $x\in\mathbb{Z}^2$, c.f.\ (\ref{nonnestabi}). Note that $\mathbb{E}[a(\theta,x)]=0$.

As before, take $N=nm$. By H\"older's inequality,
\begin{align}
\mathbb{E}[\bar{W}_N(\theta,\cdot,Y)^\alpha]&=\mathbb{E}\left[(\bar{W}_N(\theta,\cdot,Y)g(\theta,\cdot,Y))^\alpha g(\theta,\cdot,Y)^{-\alpha}\right]\nonumber\\
&\leq\mathbb{E}\left[\bar{W}_N(\theta,\cdot,Y)g(\theta,\cdot,Y)\right]^\alpha\mathbb{E}\left[g(\theta,\cdot,Y)^{-\frac{\alpha}{1-\alpha}}\right]^{1-\alpha}.\label{firstsecondterm}
\end{align}

Let us control the second term in (\ref{firstsecondterm}). $B_j$'s are pairwise disjoint and they each have $n(2C_1\sqrt{n}+1)$ elements.
Since the environment is i.i.d.,
\begin{align}
\mathbb{E}\left[g(\theta,\cdot,Y)^{-\frac{\alpha}{1-\alpha}}\right]&=\mathbb{E}\left[\exp\left(-\frac{\alpha}{1-\alpha}\sum_{j=1}^mf_K\left(\delta_nD(B_j)\right)\right)\right]=\prod_{j=1}^m\mathbb{E}\left[\exp\left(-\frac{\alpha}{1-\alpha}f_K\left(\delta_nD(B_j)\right)\right)\right]\nonumber\\
&=\mathbb{E}\left[\exp\left(-\frac{\alpha}{1-\alpha}f_K\left(\delta_nD(B_1)\right)\right)\right]^m\leq\left(1+\mathrm{e}^{\frac{\alpha}{1-\alpha}K}\mathbb{P}\left(\delta_nD(B_1)\geq\mathrm{e}^{K^2}\right)\right)^m.\label{bateman}
\end{align}
Note that, by Chebyshev's inequality,
\begin{align*}
\mathbb{P}\left(\delta_nD(B_1)\geq\mathrm{e}^{K^2}\right)&\leq\mathrm{e}^{-2K^2}\delta_n^2\mathbb{E}\left[D(B_1)^2\right]=\mathrm{e}^{-2K^2}\delta_n^2\mathbb{E}\left[\sum_{(s,i)\in B_1}a(\theta,(s,i))^2\right]\\
&=\mathrm{e}^{-2K^2}\delta_n^2n(2C_1\sqrt{n}+1)\mathbb{E}\left[a(\theta,(0,0))^2\right]\\
&\leq\mathrm{e}^{-2K^2}\delta_n^23C_1n^{3/2}\mathbb{E}\left[a(\theta,(0,0))^2\right]
\end{align*}
since, by the i.i.d.\ assumption on the environment,
only the diagonal terms survive. Take 
\begin{equation}
\label{star1}
\delta_n=C_1^{-1/2}n^{-3/4}\,,
\end{equation}
where $C_1$ is still to be defined (and will be chosen
as in Lemma \ref{lessthanhalf}).
 Then, the RHS of (\ref{bateman}) is bounded from above by
$$\left(1+3\mathbb{E}\left[a(\theta,(0,0))^2\right]\mathrm{e}^{\frac{\alpha}{1-\alpha}K-2K^2}\right)^m\leq\left(1+12\mathrm{e}^{\frac{\alpha}{1-\alpha}K-2K^2}\right)^m\leq 2^m$$ 
as soon as
\begin{equation}
\label{star2}
 12\mathrm{e}^{\frac{\alpha}{1-\alpha}K-2K^2}\leq 1.
\end{equation} 
Recalling (\ref{asama}) and (\ref{firstsecondterm}), we see that 
\begin{equation}\label{yeniasama}
\mathbb{E}[W_N(\theta,\cdot)^\alpha]\leq2^m\sum_{Y}\mathbb{E}\left[\bar{W}_N(\theta,\cdot,Y)g(\theta,\cdot,Y)\right]^\alpha.
\end{equation}

\subsection{Estimating the expectation under the tilt ($d=1+1$)}\label{estsubsec}

For every $m\geq1$, $\theta\notin sp\{e_2\}$, $\omega\in\Omega$ and $Y\in(\mathbb{V}_2)^m$, let $N=nm$ as before. By the Markov property,
\begin{align*}
\bar{W}_N(\theta,\omega,Y)&=\sum_{x_1,\ldots,x_m\in\mathbb{Z}^2}E_o^\omega[\exp\{\langle\theta,X_N\rangle-N\log\phi(\theta)\}, X_{jn}-\lfloor jn\xi(\theta)\rfloor=x_j\in J_{y_j}\ \forall j\leq m]\\
&=\sum_{x_1,\ldots,x_m\in\mathbb{Z}^2}E_o^\omega[\exp\{\langle\theta,X_n\rangle-n\log\phi(\theta)\}, X_{n}-\lfloor n\xi(\theta)\rfloor=x_1\in J_{y_1}]\\
&\hspace{1.7cm}\times E_{x_1+\lfloor n\xi(\theta)\rfloor}^\omega[\exp\{\langle\theta,X_{n}-(x_1+\lfloor n\xi(\theta)\rfloor)\rangle-n\log\phi(\theta)\},\\
&\hspace{6.6cm}X_{n}-\lfloor2n\xi(\theta)\rfloor=x_2\in J_{y_2}]\\
&\hspace{1.7cm}\times\cdots\\
&=\sum_{x_1,\ldots,x_m\in\mathbb{Z}^2}E_o^\omega[\exp\{\langle\theta,X_n\rangle-n\log\phi(\theta)\}, X_{n}-\lfloor n\xi(\theta)\rfloor=x_1\in J_{y_1}]\\
&\hspace{1.7cm}\times E_{x_1-\sqrt{n}y_1}^{T_{\lfloor n\xi(\theta)\rfloor+\sqrt{n}y_1}\omega}[\exp\{\langle\theta,X_{n}-(x_1-\sqrt{n}y_1)\rangle-n\log\phi(\theta)\},\\
&\hspace{4.5cm}\left.X_{n}-\lfloor n\xi(\theta)\rfloor=x_2-\sqrt{n}y_1\in J_{y_2}-\sqrt{n}y_1\right]\\
&\hspace{1.7cm}\times\cdots.
\end{align*}
Recall (\ref{cii}) and (\ref{abbrev}). It follows from the i.i.d.\ environment assumption that
\begin{align*}
&\mathbb{E}\left[\bar{W}_N(\theta,\cdot,Y)g(\theta,\cdot,Y)\right]\\
&\qquad=\sum_{x_1,\ldots,x_m}\mathbb{E}\left[E_o^\omega[\exp\{\langle\theta,X_n\rangle-n\log\phi(\theta)+f_K(\delta_nD(B_1))\},X_{n}-\lfloor n\xi(\theta)\rfloor=x_1\in J_{y_1}]\right.\\
&\hspace{2.7cm}\times E_{x_1-\sqrt{n}y_1}^{T_{\lfloor n\xi(\theta)\rfloor+\sqrt{n}y_1}\omega}\left[\exp\{\langle\theta,X_{n}-(x_1-\sqrt{n}y_1)\rangle-n\log\phi(\theta)+f_K(\delta_nD(B_1))\},\right.\\
&\hspace{8.1cm}\left.X_{n}-\lfloor n\xi(\theta)\rfloor=x_2-\sqrt{n}y_1\in J_{y_2}-\sqrt{n}y_1\right]\\
&\hspace{2.7cm}\left.\times\cdots\right]\\
&\qquad=\sum_{x_1,\ldots,x_m}E_o[\exp\{\langle\theta,X_n\rangle-n\log\phi(\theta)+f_K(\delta_nD(B_1))\},X_{n}-\lfloor n\xi(\theta)\rfloor=x_1\in J_{y_1}]\\
&\hspace{2.2cm}\times E_{x_1-\sqrt{n}y_1}\left[\exp\{\langle\theta,X_{n}-(x_1-\sqrt{n}y_1)\rangle-n\log\phi(\theta)+f_K(\delta_nD(B_1))\},\right.\\
&\hspace{6.8cm}\left.X_{n}-\lfloor n\xi(\theta)\rfloor=x_2-\sqrt{n}y_1\in J_{y_2}-\sqrt{n}y_1\right]\\
&\hspace{2.2cm}\times\cdots\\
&\qquad\leq E_o[\exp\{\langle\theta,X_n\rangle-n\log\phi(\theta)+f_K(\delta_nD(B_1))\},X_{n}-\lfloor n\xi(\theta)\rfloor\in J_{y_1}]\\
&\hspace{1.1cm}\times\max_{x_1\in J_{y_1}}E_{x_1-\sqrt{n}y_1}\left[\exp\{\langle\theta,X_{n}-(x_1-\sqrt{n}y_1)\rangle-n\log\phi(\theta)+f_K(\delta_nD(B_1))\},\right.\\
&\hspace{8.7cm}\left.X_{n}-\lfloor n\xi(\theta)\rfloor\in J_{y_2}-\sqrt{n}y_1\right]\\
&\hspace{1.1cm}\times\cdots\\
&\qquad=E_o[\exp\{\langle\theta,X_n\rangle-n\log\phi(\theta)+f_K(\delta_nD(B_1))\},X_{n}-\lfloor n\xi(\theta)\rfloor\in J_{y_1}]\\
&\hspace{1.1cm}\times\max_{x_1\in J_o}E_{x_1}\left[\exp\{\langle\theta,X_{n}-x_1\rangle-n\log\phi(\theta)+f_K(\delta_nD(B_1))\},X_{n}-\lfloor n\xi(\theta)\rfloor\in J_{y_2-y_1}\right]\\
&\hspace{1.1cm}\times\cdots.
\end{align*}
Plugging this in (\ref{yeniasama}), we conclude that
$$\mathbb{E}[W_N(\theta,\cdot)^\alpha]\leq\left(2\sum_{y\in\mathbb{V}_2}\max_{x\in J_o}E_x\left[\exp\{\langle\theta,X_{n}-x\rangle-n\log\phi(\theta)+f_K(\delta_nD(B_1))\},X_{n}-\lfloor n\xi(\theta)\rfloor\in J_y\right]^\alpha\right)^m.$$
The RHS of this inequality decays exponentially in $m$ if the term in the parentheses is strictly less than $1$. Since $N=nm$ and $n$ was fixed, 
this proves Lemma \ref{fracmomlemma} (and hence Theorem \ref{QneqAst}), provided that we have
\begin{lemma}\label{lessthanhalf}
Assume (\ref{cakmaeminabi}) and (\ref{space-time}). If $d=1+1$, $\alpha\in(0,1)$, $\theta\notin sp\{e_2\}$ and $\delta_n=C_1^{-1/2}n^{-3/4}$, then
\begin{equation}\label{lemansam}
\sum_{y\in\mathbb{V}_2}\max_{x\in J_o}E_x\left[\exp\{\langle\theta,X_{n}-x\rangle-n\log\phi(\theta)+f_K(\delta_nD(B_1))\},X_{n}-\lfloor n\xi(\theta)\rfloor\in J_y\right]^\alpha<1/2
\end{equation}
whenever $n$, $K$ and $C_1$ are sufficiently large.
\end{lemma}
\noindent
(The proof is valid with 
the constant $1/2$ replaced by any arbitrarily small 
positive number.)

\subsection{Finishing the proof of Theorem \ref{QneqAst}}\label{finsubsec}

It remains to give the
\begin{proof}[Proof of Lemma \ref{lessthanhalf}]
We write the sum in (\ref{lemansam}) as
\begin{equation}\label{twosums}
\sum_{y\in\mathbb{V}_2}\max_{x\in J_o}E_x\left[\cdots\right]^\alpha=\sum_{{y\in\mathbb{V}_2:}\atop{|y|>R}}\max_{x\in J_o}E_x\left[\cdots\right]^\alpha+\sum_{{y\in\mathbb{V}_2:}\atop{|y|\leq R}}\max_{x\in J_o}E_x\left[\cdots\right]^\alpha
\end{equation} with some large constant $R$,
to be determined. Since $f_K(u)=-K\one_{u\geq\mathrm{e}^{K^2}}\leq0$, the first sum on the RHS of (\ref{twosums}) is bounded from above by
\begin{align}
&\sum_{{y\in\mathbb{V}_2:}\atop{|y|>R}}\max_{x\in J_o}E_x\left[\exp\{\langle\theta,X_{n}-x\rangle-n\log\phi(\theta)\},|X_{n}-\lfloor n\xi(\theta)\rfloor-\sqrt{n}y|\leq\frac{\sqrt{n}}{2}\right]^\alpha\nonumber\\
&\qquad\leq\sum_{{y\in\mathbb{V}_2:}\atop{|y|>R}}E_o\left[\exp\{\langle\theta,X_{n}\rangle-n\log\phi(\theta)\},\left|\frac{X_{n}-\lfloor n\xi(\theta)\rfloor}{\sqrt{n}}-y\right|\leq1\right]^\alpha.\label{CLTcandidate}
\end{align}

Consider a tilted space-time walk on $\mathbb{Z}^2$ (in a deterministic environment) with transition probabilities $q^\theta(z):=q(z)\exp\{\langle\theta,z\rangle-\log\phi(\theta)\}$ for $z\in\mathcal{R}_{st}$. Let $\hat{P}_o^\theta$ denote the probability measure it induces on paths. Note that the LLN velocity under $\hat{P}_o^\theta$ is
$$\sum_{z\in\mathcal{R}_{st}}zq(z)\exp\{\langle\theta,z\rangle-\log\phi(\theta)\}=\nabla\log\phi(\theta)=\xi(\theta).$$
With this notation, (\ref{CLTcandidate}) is equal to
$$\sum_{{y\in\mathbb{V}_2:}\atop{|y|>R}}
\hat{P}_o^\theta\left(\left|\frac{X_{n}-\lfloor n\xi(\theta)\rfloor}
{\sqrt{n}}-y\right|\leq1\right)^\alpha
\leq
\sum_{{y\in\mathbb{V}_2:}\atop{|y|>R}}
\hat{P}_o^\theta\left(\left|\frac{X_{n}-\lfloor n\xi(\theta)\rfloor}
{\sqrt{n}}\right|\geq |y|-1\right)^\alpha
$$
which, by 
Chebyshev's inequality,
can be made arbitrarily small (uniformly in large $n$) 
by choosing $R$ sufficiently large.

The second sum on the RHS of (\ref{twosums}) is bounded from above by
$$(2R+1)\max_{x\in J_o}E_x\left[\exp\{\langle\theta,X_{n}-x\rangle-n\log\phi(\theta)+f_K(\delta_nD(B_1))\}\right]^\alpha.$$
Therefore, to conclude the proof of 
Lemma \ref{lessthanhalf}, it suffices to show that
\begin{equation}\label{lastsuff}
E_o\left[\exp\{\langle\theta,X_{n}\rangle-n\log\phi(\theta)+f_K(\delta_nD(B_1-x))\}\right]\leq \left(\frac{1}{8R}\right)^{\alpha^{-1}}
\end{equation}
for every $x\in J_o$.

Similar to $B_1$ defined in (\ref{bicey}), introduce a new set $$\bar{B}_1:=\{(s,i)\in\mathbb{Z}^2: 0\leq i<n, |(s,i)-\lfloor i\xi(\theta)\rfloor|\leq (C_1-1/2)\sqrt{n}\}.$$ Note that $\bar{B}_1\subset B_1-x$ for every $x\in J_o$ since $|x|\leq\sqrt{n}/2$.

\begin{align}
&E_o\left[\exp\{\langle\theta,X_{n}\rangle-n\log\phi(\theta)+f_K(\delta_nD(B_1-x))\}\right]\nonumber\\
&\qquad=\mathrm{e}^{-K}E_o\left[\exp\{\langle\theta,X_{n}\rangle-n\log\phi(\theta)\},\delta_nD(B_1-x)\geq\mathrm{e}^{K^2}\right]\nonumber\\
&\qquad\qquad+E_o\left[\exp\{\langle\theta,X_{n}\rangle-n\log\phi(\theta)\},\{X_i: 0\leq i<n\}\not\subset\bar{B}_1,\delta_nD(B_1-x)<\mathrm{e}^{K^2}\right]\nonumber\\
&\qquad\qquad+E_o\left[\exp\{\langle\theta,X_{n}\rangle-n\log\phi(\theta)\},\{X_i: 0\leq i<n\}\subset\bar{B}_1,\delta_nD(B_1-x)<\mathrm{e}^{K^2}\right]\nonumber\\
&\qquad\leq\mathrm{e}^{-K} + \hat{P}_o^\theta\left(\{X_i: 0\leq i<n\}\not\subset\bar{B}_1\right)\label{sirabunda}\\
&\qquad\qquad+E_o\left[\exp\{\langle\theta,X_{n}\rangle-n\log\phi(\theta)\},\{X_i: 0\leq i<n\}\subset\bar{B}_1,\delta_nD(B_1-x)<\mathrm{e}^{K^2}\right].\nonumber
\end{align}
The first term in (\ref{sirabunda}) is small when $K$ is large. Donsker's invariance principle ensures that the second term can be made arbitrarily small (uniformly in $n$) by choosing $C_1$ sufficiently large.

Let us focus on the third term in (\ref{sirabunda}). For any sequence $(A_n)_{n\geq1}$ of natural numbers,
\begin{align}
&E_o[\exp\{\langle\theta,X_{n}\rangle-n\log\phi(\theta)\},\{X_i: 0\leq i<n\}\subset\bar{B}_1,\delta_nD(B_1-x)<\mathrm{e}^{K^2}]\nonumber\\
&\qquad\leq E_o[\exp\{\langle\theta,X_{n}\rangle-n\log\phi(\theta)\},\{X_i: 0\leq i<n\}\subset\bar{B}_1,\delta_n\!\!\!\!\!\sum_{{(s,i)\in B_1-x}\atop{(s,i)\neq X_i}}\!\!\!a(\theta,(s,i))<-A_n]\nonumber\\
&\qquad\quad+E_o[\exp\{\langle\theta,X_{n}\rangle-n\log\phi(\theta)\},\{X_i: 0\leq i<n\}\subset\bar{B}_1,\delta_n\sum_{i=0}^{n-1}a(\theta,X_i)<\mathrm{e}^{K^2}+A_n]\nonumber\\
&\qquad\leq\sum_{x_1,\ldots,x_{n-1}}\!\!\!\!\mathbb{E}[E_o^\omega\left[\exp\{\langle\theta,X_{n}\rangle-n\log\phi(\theta)\}, X_i=x_i\ \forall i<n\right],\delta_n\!\!\!\!\!\sum_{{(s,i)\in B_1-x}\atop{(s,i)\neq x_i}}\!\!\!a(\theta,(s,i))<-A_n]\nonumber\\
&\qquad\quad+E_o[\exp\{\langle\theta,X_{n}\rangle-n\log\phi(\theta)\},\delta_n\sum_{i=0}^{n-1}a(\theta,X_i)<\mathrm{e}^{K^2}+A_n]\nonumber\\
&\qquad=\sum_{x_1,\ldots,x_{n-1}}\!\!\!\!E_o\left[\exp\{\langle\theta,X_{n}\rangle-n\log\phi(\theta)\}, X_i=x_i\ \forall i<n\right]\times\mathbb{P}(\delta_n\!\!\!\!\!\sum_{{(s,i)\in B_1-x}\atop{(s,i)\neq x_i}}\!\!\!a(\theta,(s,i))<-A_n)\label{indepeq}\\
&\qquad\quad+E_o[\exp\{\langle\theta,X_{n}\rangle-n\log\phi(\theta)\},\delta_n\sum_{i=0}^{n-1}a(\theta,X_i)<\mathrm{e}^{K^2}+A_n]\nonumber\\
&\qquad\leq\max_{x_1,\ldots,x_{n-1}}\mathbb{P}(\delta_n\!\!\!\!\!\sum_{{(s,i)\in B_1-x}\atop{(s,i)\neq x_i}}\!\!\!a(\theta,(s,i))<-A_n)\nonumber\\
&\qquad\quad+E_o[\exp\{\langle\theta,X_{n}\rangle-n\log\phi(\theta)\},\delta_n\sum_{i=0}^{n-1}a(\theta,X_i)<\mathrm{e}^{K^2}+A_n]\nonumber\\
&\qquad\leq A_n^{-2}\delta_n^22C_1n^{3/2}\mathbb{E}\left[a(\theta,(0,0))^2\right]\label{chebyineq}\\
&\qquad\quad+E_o[\exp\{\langle\theta,X_{n}\rangle-n\log\phi(\theta)\},\delta_n\sum_{i=0}^{n-1}a(\theta,X_i)<\mathrm{e}^{K^2}+A_n].\nonumber
\end{align}
Here, (\ref{indepeq}) follows from the independence assumption on the environment, and (\ref{chebyineq}) is an application of Chebyshev's inequality. Since $\delta_n=C_1^{-1/2}n^{-3/4}$, the first term in (\ref{chebyineq}) goes to zero as $n\to\infty$ if $A_n\to\infty$. 

Choose $A_n$ such that $A_n\to\infty$ and $A_n=o(n^{1/4})$ as $n\to\infty$. For any $\mu\in\mathbb{R}^+$, the second term in (\ref{chebyineq}) is equal to
\begin{align}
&E_o[\exp\{\langle\theta,X_{n}\rangle-n\log\phi(\theta)\},\delta_n\sum_{i=0}^{n-1}(a(\theta,X_i)-\mu)<\mathrm{e}^{K^2}+A_n-\mu n\delta_n]\nonumber\\
&\quad\leq M_nE_o[\exp\{\langle\theta,X_{n}\rangle-n\log\phi(\theta)\}\sum_{i=0}^{n-1}(a(\theta,X_i)-\mu)^2]\label{crossterms}\\
&\quad\quad+M_nE_o[\exp\{\langle\theta,X_{n}\rangle-n\log\phi(\theta)\}\sum_{i\neq j}(a(\theta,X_i)-\mu)(a(\theta,X_j)-\mu)]\nonumber
\end{align}
by Chebyshev's inequality, where $M_n=\left(\frac{\delta_n}{\mu n\delta_n-A_n-\mathrm{e}^{K^2}}\right)^2=O(n^{-2})$.

By the FKG inequality (c.f.\ \cite{GrimmettPercolation}),
\begin{align}
E_o\left[\exp\{\langle\theta,X_1\rangle-\log\phi(\theta)\}a(\theta,(0,0))\right]&=\mathbb{E}\left[E_o^\omega\left[\exp\{\langle\theta,X_1\rangle-\log\phi(\theta)\}\right]\,a(\theta,(0,0))\right]\label{FKG}\\
&>\mathbb{E}\left[E_o^\omega\left[\exp\{\langle\theta,X_1\rangle-\log\phi(\theta)\}\right]\right]\mathbb{E}\left[a(\theta,(0,0))\right]=0\nonumber
\end{align}
since $E_o^\omega\left[\exp\{\langle\theta,X_1\rangle-\log\phi(\theta)\}\right]$ and $a(\theta,(0,0))$ are easily checked to be either both strictly increasing functions (when $\langle\theta,e_1\rangle>0$) or both strictly decreasing functions (when $\langle\theta,e_1\rangle<0$) of the random variable $\pi((0,0),(1,1))$. If we choose
%OO1 
\begin{equation}
\label{eq-150310a}
\mu=E_o\left[\exp\{\langle\theta,X_1\rangle-\log\phi(\theta)\}
a(\theta,(0,0))\right],
\end{equation}
then the second term in (\ref{crossterms}) vanishes by the independence assumption on the environment. Finally, observe that the first term in (\ref{crossterms}) is equal to
$$nM_nE_o\left[\exp\{\langle\theta,X_1\rangle-\log\phi(\theta)\}(a(\theta,(0,0))-\mu)^2\right]=O(n^{-1}).\qedhere$$
\end{proof}

\subsection{Proof of Theorem \ref{QneqAstyeni}}\label{ayyy}

Let us recall a few points regarding the arguments in Subsections \ref{tiltsubsec} -- \ref{finsubsec}. There, since $d=1+1$, the volume of $B_1$ (defined in (\ref{bicey})) is $O(n^{3/2})$. The variance of $D(B_1)$ (c.f.\ (\ref{abbrev})) scales like that volume. We take $\delta_n=O(n^{-3/4})$ so that the variance of $\delta_nD(B_1)$ is $O(1)$. With this choice, $n\delta_n\to\infty$ as $n\to\infty$. As we saw, this fact is crucial in the proof of Theorem \ref{QneqAst}.

In this subsection, we will assume that $d=2+1$. For every $m\geq1$, $1\leq j\leq m$, $\theta\notin sp\{e_3\}$, $C_1\geq1$ and $Y=(y_1,\ldots,y_m)\in(\mathbb{V}_3)^m$, we define 
\begin{equation}\label{biceyyeni}
B_j:=\{(r,k): r\in\mathbb{Z}^2, (j-1)n\leq k<jn, \left|(r,k)-\lfloor k\xi(\theta)\rfloor-\sqrt{n}y_{j-1}\right|\leq C_1\sqrt{n}\},
\end{equation} similar to (\ref{bicey}). Note that the volume of this new set is $O(n^2)$. If we were to define $D(B_1)$ 
analogously to (\ref{abbrev}), then we would have to take $\delta_n\leq O(n^{-1})$ in order to make the variance of $\delta_nD(B_1)$ not grow with $n$,
in which case $n\delta_n$ remains bounded. 
Hence, the proof for $d=1+1$ does not directly carry over to the case $d=2+1$.

To resolve this issue, following
\cite{Lacoin},
we will modify the proof by redefining $D(B_1)$ and $\delta_n$. (We will continue using these names so that we can refer to the parts of Subsections \ref{tiltsubsec} -- \ref{finsubsec} that carry over word by word.) The modification
amounts essentially to using a tilting that is quadratic, instead
of linear, in the local drift, as follows.

For every $(r,k)$ and $(s,l)$ with $r,s\in\mathbb{Z}^2$ and $k,l\geq1$, let 
\begin{equation}\label{quadtilt}
V((r,k),(s,l)):=\frac{1}{|k-l|}\one_{\{\left|(s,l)-(r,k)-\lfloor (l-k)\xi(\theta)\rfloor\right|< C_2\sqrt{|k-l|}\}}
\end{equation}
if $k\neq l$, and set it to be equal to zero if $k=l$. Here, the constant $C_2\geq1$ will be determined later. 
Given any $n$ integer and
$x_1,\ldots,x_n\in\mathbb{Z}^3$ with 
$\langle x_k,e_3\rangle=k$, it follows easily from (\ref{quadtilt}) that
\begin{align}
\mbox{\rm for any $s\in\mathbb{Z}^2$, $l\in\{1,\ldots,n\}$},&\qquad\sum_{k=1}^nV(x_k,(s,l))\leq2\log n,\nonumber\\
\sum_{k=1}^n\sum_{(s,l)\in B_1}\!\!\!\!V(x_k,(s,l))&\leq\sum_{{1\leq k,l\leq n}\atop{k\neq l}}\frac{1}{|k-l|}\left(2C_2\sqrt{|k-l|}\right)^2\leq4C_2^2n^2,\nonumber\\
\sum_{(s,l)\in B_1}\left(\sum_{k=1}^nV(x_k,(s,l))\right)^2&=\left(\max_{(s',l')}\sum_{k=1}^nV(x_k,(s',l'))\right)\sum_{(s,l)\in B_1}\sum_{k=1}^nV(x_k,(s,l))\nonumber\\&\leq(2\log n)(4C_2^2n^2)=8C_2^2n^2\log n,\quad\mbox{and}\label{coni}\\
\sum_{{(r,k)\in B_1,}\atop{(s,l)\in B_1}}V((r,k),(s,l))^2&\leq\sum_{k=1}^n(2C_1\sqrt{n})^2\sum_{l=1}^n\frac{\one_{\{k\neq l\}}}{|k-l|^2}(2C_2\sqrt{|k-l|})^2\nonumber\\
&=16C_1^2C_2^2n\sum_{{1\leq k,l\leq n}\atop{k\neq l}}\frac1{|k-l|}
\leq32C_1^2C_2^2n^2\log n.\label{boni}
\end{align}

Recall the tilted law $\hat{P}_o^\theta$ introduced in the proof
of Lemma \ref{lessthanhalf}.
\begin{lemma}\label{gereklilemmabir}
For any $\delta>0$, there exists a $C_2\geq1$ such that $\nu(n,X):=\sum_{1\leq i,j\leq n}V(X_i,X_j)$ 
satisfies
$$\hat{P}_o^{\theta}(\nu(n,X)<n\log (n-1)/2)\leq\delta$$ for every $n\geq2$.
\end{lemma}
\begin{proof}
For any realization of $X=(X_i)_{i\geq1}$, $$\nu(n,X)\leq\sum_{{1\leq i,j\leq n}\atop{i\neq j}}\frac{1}{|i-j|}=:H(n).$$
Observe that $$\hat{E}_o^{\theta}[\nu(n,X)]=\sum_{1\leq i,j\leq n}\hat{E}_o^{\theta}[V(X_i,X_j)]=\sum_{{1\leq i,j\leq n}\atop{i\neq j}}\frac{1}{|i-j|}\hat{P}_o^{\theta}(|X_i-X_j-\lfloor (i-j)\xi(\theta)\rfloor|< C_2\sqrt{|i-j|}).$$
When $C_2$ is sufficiently large, the CLT implies that $$\hat{P}_o^{\theta}(|X_i-X_j-\lfloor (i-j)\xi(\theta)\rfloor|< C_2\sqrt{|i-j|})\geq(1-\delta/2)$$ for any $i\neq j$. Therefore, $\hat{E}_o^{\theta}[\nu(n,X)]\geq(1-\delta/2)H(n)$. Applying Markov's inequality, we see that $$\hat{P}_o^{\theta}(\nu(n,X)<H(n)/2)=\hat{P}_o^{\theta}(H(n)-\nu(n,X)>H(n)/2)\leq\delta.$$ This implies the desired result 
since 
$H(n)\geq n\log (n-1)$.
\end{proof}
For any $\theta\in\mathbb{R}^3$ and $x\in\mathbb{Z}^3$, define $a(\theta,x):=\langle\theta,v(T_x\omega)-\xi_o\rangle$ as before, where $v(\omega)=\sum_{z\in\mathcal{R}}{\pi(0,z)}z$.
\begin{lemma}\label{gereklilemmaiki}
There exists a $\beta>0$ such that 
$$\mu:=E_o\left[\exp\{\langle\theta,X_1\rangle-\log\phi(\theta)\}
%OO1
a(\theta,(0,0,0))\right]>0$$ 
whenever $\mathrm{dist}(\theta,sp\{e_3\})\in(0,\beta)$.
\end{lemma}
\begin{proof}
For every $\theta\notin sp\{e_3\}$, let $$F(\theta):=\mathbb{E}\{E_o^\omega[\mathrm{e}^{\langle\theta,X_1\rangle}]E_o^\omega[\langle\theta,X_1\rangle]\}\quad\mbox{and}\quad G(\theta):=E_o[\mathrm{e}^{\langle\theta,X_1\rangle}]E_o[\langle\theta,X_1\rangle]=\phi(\theta)\langle\theta,\xi_o\rangle.$$
Our aim is to show that $F(\theta)>G(\theta)$.

Write $\theta=ce_3+\theta'$ for some $c\in\mathbb{R}$ and $\theta'\in\mathbb{R}^3$ such that $\langle\theta',e_3\rangle=0$. Then, $F(\theta)=\mathrm{e}^cF(\theta')+c\mathrm{e}^c\phi(\theta')$ and $G(\theta)=\mathrm{e}^cG(\theta')+c\mathrm{e}^c\phi(\theta')$. Therefore, it suffices to show that $F(\theta')>G(\theta')$.

Clearly, we have $$\left.\nabla F(\theta)\right|_{\theta=0}=\left.\nabla G(\theta)\right|_{\theta=0}=E_o[X_1]=\xi_o.$$
Also, for any $u,u'\in\mathbb{R}^3$,
with $D^2F$ denoting the Hessian of $F$,
$$\left.\langle u,D^2F(\theta)u'\rangle\right|_{\theta=0}=2\mathbb{E}\{E_o^\omega[\langle X_1,u\rangle]E_o^\omega[\langle X_1,u'\rangle]\}$$
and
$$\left.\langle u,D^2G(\theta)u'\rangle\right|_{\theta=0}=2E_o[\langle X_1,u\rangle]E_o[\langle X_1,u'\rangle]=2\langle\xi_o,u\rangle\langle\xi_o,u'\rangle.$$
By Schwarz' inequality (which is strict since the walk is uniformly elliptic in the directions other than $e_3$), $$\inf_{{|u|=1}\atop{\langle u,e_3\rangle=0}}\left(\left.\langle u,D^2F(\theta)u\rangle\right|_{\theta=0}-\left.\langle u,D^2G(\theta)u\rangle\right|_{\theta=0}\right)>0.$$
Finally, Taylor's theorem implies the existence of a $\beta>0$ such that $F(\theta')>G(\theta')$ whenever $|\theta'|\in(0,\beta)$.
\end{proof}

Now, we are ready to give the new definition of $D(B_1)$ which is suitable for $d=2+1$. For any $\theta\in\mathbb{R}^3$ such that $\mathrm{dist}(\theta,sp\{e_3\})\in(0,\beta)$ (with $\beta$ as in Lemma \ref{gereklilemmaiki}), let
\begin{equation}\label{yeniD}
D(B_1):=\sum_{{(r,k)\in B_1,}\atop{(s,l)\in B_1}}V((r,k),(s,l))a(\theta,(r,k))a(\theta,(s,l)).
\end{equation}
Note that $V((\cdot,k),(\cdot,k))=0$ for every $1\leq k\leq n$. Since $\mathbb{E}[a(\theta,0)]=0$, it follows from the independence of the environment that $\mathbb{E}[D(B_1)]=0$. Also, $\mathbb{E}[D(B_1)^2]\leq1024|\theta|^4C_1^2C_2^2n^2\log n$ by (\ref{boni}) and the fact that $|a(\theta,0)|\leq 2|\theta|$.

If we choose
$$\delta_n:=n^{-1}(\log n)^{-1/2},$$ then the variance of $\delta_nD(B_1)$ is $O(1)$. Once we have this fact, the arguments in Subsections \ref{tiltsubsec} -- \ref{finsubsec} carry over until (\ref{lastsuff}). So, it suffices to show that $E_o[\exp\{\langle\theta,X_{n}\rangle-n\log\phi(\theta)\},\delta_nD(B_1-x)<\mathrm{e}^{K^2}]$ is small for all $x\in J_o$ when $n$ 
and $K$ are large. In the estimate below, we will (WLOG) take $x=0$.

Let $\gamma=1/2$, and observe that
\begin{align}
&E_o[\exp\{\langle\theta,X_{n}\rangle-n\log\phi(\theta)\},\delta_nD(B_1)<\mathrm{e}^{K^2}]\nonumber\\
&\qquad\leq E_o[\exp\{\langle\theta,X_{n}\rangle-n\log\phi(\theta)\},\nu(n,X)\geq\gamma n\log (n-1), \delta_nD(B_1)<\mathrm{e}^{K^2}]\nonumber\\
&\qquad\quad+E_o\left[\exp\{\langle\theta,X_{n}\rangle-n\log\phi(\theta)\},\nu(n,X)<\gamma n\log (n-1)\right]\nonumber\\
&\qquad=E_o[\exp\{\langle\theta,X_{n}\rangle-n\log\phi(\theta)\},\nu(n,X)\geq\gamma n\log (n-1),\nonumber\\
&\hspace{1.7cm}\delta_n(D(B_1)-\mu^2\nu(n,X))<\mathrm{e}^{K^2}-\mu^2\delta_n\nu(n,X)]
+\hat{P}_o^{\theta}(\nu(n,X)<\gamma n\log (n-1))\nonumber\\
&\qquad\leq M_nE_o[\exp\{\langle\theta,X_{n}\rangle-n\log\phi(\theta)\}(D(B_1)-\mu^2\nu(n,X))^2,\nu(n,X)\geq\gamma n\log (n-1)]\label{cebican}\\
&\qquad\quad+\hat{P}_o^{\theta}(\nu(n,X)<\gamma n\log (n-1))\nonumber\\
&\qquad\leq M_nE_o[\exp\{\langle\theta,X_{n}\rangle-n\log\phi(\theta)\}(D(B_1)-\mu^2\nu(n,X))^2]
+\hat{P}_o^{\theta}(\nu(n,X)<\gamma n\log (n-1)).\label{kahan}
\end{align}
Here, (\ref{cebican}) follows from 
the elementary inequality $\one_{a<b}\leq a^2/b^2$ with
$a=\delta_n(D(B_1)-\mu^2\nu(n,X))$ and $b=\mathrm{e}^{K^2}-\mu^2\delta_n\nu(n,X)<0$,
and
\begin{equation}\label{emen}
M_n=\left(\frac{\delta_n}{\mu^2\delta_n\gamma n\log (n-1)
-\mathrm{e}^{K^2}}\right)^2.
\end{equation}Choose $C_2$ sufficiently large so that the second term in (\ref{kahan}) is small for all $n\geq2$ by Lemma \ref{gereklilemmabir}.

It remains to control the first term in (\ref{kahan}). Note that
\begin{align*}
&D(B_1)-\mu^2\nu(n,X)\\
&\qquad=2\mu\sum_{k=1}^n\sum_{(s,l)\in B_1}V(X_k,(s,l))(a(\theta,(s,l))-\mu\one_{\{X_l=(s,l)\}})\\
&\qquad\quad+\sum_{{(r,k)\in B_1,}\atop{(s,l)\in B_1}}V((r,k),(s,l))(a(\theta,(r,k))-\mu\one_{\{X_k=(r,k)\}})(a(\theta,(s,l))-\mu\one_{\{X_l=(s,l)\}}),
\end{align*}
and
\begin{align}
&E_o[\exp\{\langle\theta,X_{n}\rangle-n\log\phi(\theta)\}(D(B_1)-\mu^2\nu(n,X))^2]\nonumber\\
&\qquad\leq8\mu^2E_o\left[\exp\{\cdots\}\left(\sum_{k=1}^n\sum_{(s,l)\in B_1}V(X_k,(s,l))(a(\theta,(s,l))-\mu\one_{\{X_l=(s,l)\}})\right)^2\right]\nonumber\\
&\qquad\quad+2E_o\left[\exp\{\cdots\}\left(\sum_{{(r,k)\in B_1,}\atop{(s,l)\in B_1}}\!\!\!\!V((r,k),(s,l))(a(\theta,(r,k))-\mu\one_{\{X_k=(r,k)\}})(a(\theta,(s,l))-\mu\one_{\{X_l=(s,l)\}})\right)^2\right]\nonumber\\
&\qquad=:8\mu^2\mathfrak{E}_1 + 2\mathfrak{E}_2\label{yervarmi}
\end{align}
by the inequality $(a+b)^2\leq 2(a^2+b^2)$.

One should note at this stage that in fact, even 
though $\mu$ was chosen to equal the mean under the tilted measure
 of $a(\theta,0)$,
it is not necessarily the case
that the mean of
$D(B_1)-\mu^2 \nu(n,X)$ under that  tilted measure
vanishes. This makes the control of $\mathfrak{E}_i$ somewhat messy, involving 
a local CLT (Lemma \ref{localCLTrefined}).

We turn to the details of the computation.
$\mathfrak{E}_1$ can be written as a double sum over pairs $(s,l), (s',l')\in B_1$. If $(s,l)\neq(s',l')$, then it is clear from independence that this pair does not contribute to $\mathfrak{E}_1$ on the event $\{X_l\neq(s,l)\}\cup\{X_{l'}\neq(s',l')\}$. Therefore,
\begin{align}
\mathfrak{E}_1&=E_o\left[\exp\{\cdots\}\sum_{(s,l)\in B_1}\left(\sum_{k=1}^nV(X_k,(s,l))(a(\theta,(s,l))-\mu\one_{\{X_l=(s,l)\}})\right)^2\right]\nonumber\\
&\quad+\sum_{{k,k',l,l':}\atop{l\neq l'}}E_o\left[\exp\{\langle\theta,X_{n}\rangle-n\log\phi(\theta)\}V(X_k,X_l)V(X_{k'},X_{l'})(a(\theta,X_l)-\mu)(a(\theta,X_{l'})-\mu)\right]\nonumber\\
&=:E_o\left[\exp\{\cdots\}\sum_{(s,l)\in B_1}\left(\sum_{k=1}^nV(X_k,(s,l))(a(\theta,(s,l))-\mu\one_{\{X_l=(s,l)\}})\right)^2\right]+\sum_{{k,k',l,l':}\atop{l\neq l'}}\mathfrak{E}_1(k,k',l,l')\nonumber\\
&\leq(2|\theta|+\mu)^2E_o\left[\exp\{\cdots\}\sum_{(s,l)\in B_1}\left(\sum_{k=1}^nV(X_k,(s,l))\right)^2\right]+\sum_{{k,k',l,l':}\atop{l\neq l'}}\mathfrak{E}_1(k,k',l,l')\nonumber\\
&\leq(2|\theta|+\mu)^28C_2^2n^2\log n+\sum_{{k,k',l,l':}\atop{l\neq l'}}\mathfrak{E}_1(k,k',l,l')\label{sanatsanat}
\end{align}
by (\ref{coni}) and the fact that $|a(\theta,\cdot)|\leq2|\theta|$.

If $l'> \max(k,k',l)$, 
then $\mathfrak{E}_1(k,k',l,l')$ is equal to zero since we can condition on the path up to $l'$ and use the fact that, for any $(x_i)_1^{l'}$, $$E_o\left[\exp\{\langle\theta,X_n-X_{l'}\rangle-(n-l')\log\phi(\theta)\}(a(\theta,X_{l'})-\mu)\,\left|\,(X_i)_1^{l'}=(x_i)_1^{l'}\right.\right]=0$$ by the definition of $\mu$, c.f.\ Lemma \ref{gereklilemmaiki}.

If $l<l'<k'<k$, then $V(X_k,X_l)$ and $V(X_{k'},X_{l'})$ create a slight complication since $X_k$ and $X_{k'}$ are not independent of $X_{l'+1}-X_{l'}$. Indeed, 
\begin{align*}
\mathfrak{E}_1(k,k',l,l')&=\sum_{{x_1,\ldots,x_{l'}}\atop{z\in\mathcal{R}}}E_o\left[\exp\{\cdots\}(a(\theta,X_l)-\mu)(a(\theta,X_{l'})-\mu), (X_i)_1^{l'}=(x_i)_1^{l'}, X_{l'+1}-X_{l'}=z\right]\\
&\qquad\qquad\times \hat E_o^\theta[V(X_k,x_l)V(X_{k'},x_{l'})\,|\,X_{l'+1}=x_{l'}+z],
\end{align*}
and the latter expectation depends on $z$. (If it were independent of $z$, we could simply take the sum over $z\in\mathcal{R}$ and conclude that $\mathfrak{E}_1(k,k',l,l')=0$.) However, for any $z,z'\in\mathcal{R}$,
\begin{align*}
&\left|\hat E_o^\theta[V(X_k,x_l)V(X_{k'},x_{l'})\,|\,X_{l'+1}=x_{l'}+z] - \hat E_o^\theta[V(X_k,x_l)V(X_{k'},x_{l'})\,|\,X_{l'+1}=x_{l'}+z']\right|\\
&\qquad\leq\!\!\sum_{{x_{k'}:}\atop{V(x_{k'},x_{l'})>0}}\!\!(k'-l')^{-1}\left|\hat P_o^\theta(X_{k'}=x_{k'}|X_{l'+1}=x_{l'}+z) - \hat P_o^\theta(X_{k'}=x_{k'}|X_{l'+1}=x_{l'}+z')\right|\\
&\qquad\qquad\qquad\qquad\times \hat E_o^\theta[V(X_k,x_l)\,|\,X_{k'}=x_{k'}]\\
&\qquad\leq 4C_2^2(k'-l')(k'-l')^{-1}O((k'-l')^{-3/2})(k-l)^{-1}
= O((k-l)^{-1}(k'-l')^{-3/2})
\end{align*}
uniformly in $(x_i)_1^{l'}$, c.f. Lemma \ref{localCLTrefined} (given below). Hence, 
$$\sum_{l<l'<k'<k}\mathfrak{E}_1(k,k',l,l')\leq O(n^2\log n).$$

It is easy to see that this technique works for $\mathfrak{E}_1(k,k',l,l')$ in all other cases, and we get $\mathfrak{E}_1\leq O(n^2\log n)$ by (\ref{sanatsanat}).

$\mathfrak{E}_2$ is a quadruple sum over $(r,k), (r',k'), (s,l), (s',l')\in B_1$ that is symmetric in $(r,k)$ and $(s,l)$ (as well as in $(r',k')$ and $(s',l')$). Recall that $V((r,k),(s,l))=0$ when $(r,k)=(s,l)$. If $(r,k)\notin\{(r',k'),(s',l')\}$, then it is clear from independence that there is no contribution to $\mathfrak{E}_2$ on the event $\{X_k\neq(r,k)\}$. The contribution from the complementary event can be estimated using Lemma \ref{localCLTrefined}, just like in the case of $\mathfrak{E}_1$. Putting everything together and recalling (\ref{boni}),  we see that
\begin{align*}
\mathfrak{E}_2&\leq2E_o\left[\exp\{\cdots\}\sum_{{(r,k)\in B_1,}\atop{(s,l)\in B_1}}\left(V((r,k),(s,l))(a(\theta,(r,k))-\mu\one_{\{X_k=(r,k)\}})(a(\theta,(s,l))-\mu\one_{\{X_l=(s,l)\}})\right)^2\right]\\
&\quad+O(n^2\log n)\\
&\leq2(2|\theta|+\mu)^4E_o\left[\exp\{\langle\theta,X_{n}\rangle-n\log\phi(\theta)\}\sum_{{(r,k)\in B_1,}\atop{(s,l)\in B_1}}V((r,k),(s,l))^2\right]+O(n^2\log n)\\
&\leq2(2|\theta|+\mu)^432C_1^2C_2^2n^2\log n+O(n^2\log n)\\
&\leq O(n^2\log n).
\end{align*}

Finally, $$M_nE_o[\exp\{\langle\theta,X_{n}\rangle-n\log\phi(\theta)\}(D(B_1)-\mu^2\nu(n,X))^2]\leq M_n(8\mu^2\mathfrak{E}_1 + 2\mathfrak{E}_2)\leq O((\log n)^{-1})$$
by (\ref{emen}) and (\ref{yervarmi}). This concludes the proof of Theorem \ref{QneqAstyeni}, apart from

\begin{lemma}\label{localCLTrefined}
For any $z,z'\in\mathcal{R}$,
$$\sup_{x\in\mathbb{Z}^3}|\hat P_z^\theta(X_m=x) - \hat P_{z'}^\theta(X_m=x)|\leq O(m^{-3/2})\quad\mbox{as }m\to\infty.$$
\end{lemma}

\begin{proof}
Let $G^\theta$ be the centered Gaussian density on $\mathbb{R}^2$ that has the same covariance with $(\langle X_1,e_1\rangle,\langle X_1,e_2\rangle)$ under $\hat P_o^\theta$. 
For any $z\in\mathcal{R}$, it is shown in Theorem 22.1 of \cite{BhattaRangaRao} that
$$\sup_{x}\left|\hat P_z^\theta(X_m=x) - \frac2{m}G^\theta\left(\frac{\langle x-z-m\xi(\theta),e_1\rangle}{\sqrt{m}}, \frac{\langle x-z-m\xi(\theta),e_2\rangle}{\sqrt{m}}\right)\right|\leq O(m^{-3/2})\quad\mbox{as }m\to\infty.$$
Here, the supremum is taken over all $x=(x_1,x_2,x_3)\in\mathbb{Z}^3$ such that $x_1+x_2+m+1$ is even and $x_3=m+1$. (Otherwise, $\hat P_z^\theta(X_m=x)$ is equal to zero.) Since $\sup_{y\in\mathbb{R}^2}|\nabla_y G^\theta(y)|<\infty$, the desired result follows from the triangle inequality.
\end{proof}

\section{Inequality of the rate functions for space-only RWRE}\label{SpaceOnlySection}

\subsection{Reducing to a fractional moment estimate}

Consider space-only RWRE on $\mathbb{Z}^d$ with $d\geq1$. Assume that the walk is non-nestling relative to the canonical basis vector $e_d$. By Jensen's inequality, the quenched and the averaged logarithmic moment generating functions
$$\Lambda_q(\theta):=\lim_{N\to\infty}\frac{1}{N}\log E_o^\omega\left[\exp\{\langle\theta,X_N\rangle\}\right]\quad\mbox{and}\quad\Lambda_a(\theta):=\lim_{N\to\infty}\frac{1}{N}\log E_o\left[\exp\{\langle\theta,X_N\rangle\}\right]$$
satisfy $\Lambda_q(\theta)\leq\Lambda_a(\theta)\leq|\theta|$ for every $\theta\in\mathbb{R}^d$.

Recall the definition of regeneration times $(\tau_n)_{n\geq0}$ (relative to $e_d$) given in Subsection \ref{regss}. Let $$\beta:=\inf\{i\geq0:\langle X_i,e_d\rangle<\langle X_o,e_d\rangle\}\in[1,\infty].$$
By the non-nestling assumption, there exist constants $c_2,c_3>0$ such that
\begin{equation}\label{serdar}
\mathrm{ess}\inf_{\mathbb{P}}P_o^\omega(\beta=\infty)\geq c_2\quad\mbox{and}\quad\mathrm{ess}\sup_{\mathbb{P}}P_o^\omega(\tau_1>n)\leq\mathrm{e}^{-c_3n}
\end{equation}
for every $n\geq1$, c.f.\ \cite{SznitmanSlowdown}. These bounds clearly imply that
\begin{equation}\label{yoruk}
\mathrm{ess}\sup_{\mathbb{P}}E_o^{\omega}[\left.\exp\{c\tau_1\}\right|\beta=\infty]\leq c_2^{-1}\mathrm{ess}\sup_{\mathbb{P}}E_o^{\omega}[\exp\{c\tau_1\}]=:H(c)<\infty
\end{equation} whenever $c<c_3$.

For every $c\in(0,c_3]$, introduce the set
\begin{equation}\label{metkelsi}
\mathcal{C}(c):=\{\theta\in\mathbb{R}^d:2|\theta|<c\}.
\end{equation}
\begin{lemma}\label{zikkim}
For every $\theta\in\mathcal{C}(c_3)$,
\begin{equation}\label{cevikuff}
E_o[\left.\exp\{\langle\theta,X_{\tau_1}\rangle-\Lambda_a(\theta)\tau_1\}\right|\beta=\infty]=1.
\end{equation}
$\Lambda_a$ is analytic on $\mathcal{C}(c_3)$. $\nabla\Lambda_a(0)=\xi_o$. The Hessian $\mathcal{H}_a$ of $\Lambda_a$ is positive definite on $\mathcal{C}(c_3)$. For every $c<c_3$ and $\theta\in\mathcal{C}(c)$, the smallest eigenvalue of $\mathcal{H}_a(\theta)$ is bounded from below by a positive constant that depends only on $c$ and the ellipticity constant $\kappa$ of the walk.
\end{lemma}
\begin{proof}
See the proofs of Lemmas 6 and 12 of \cite{YilmazAveraged}. In particular, the desired lower bound for the smallest eigenvalue of $\mathcal{H}_a$ is evident from equation (2.10) of that paper.
\end{proof}

Given any $N\geq1$, $\theta\in\mathcal{C}(c_3)$ and $\omega\in\Omega$, define
\begin{align*}
\hat{W}_N(\theta,\omega)&:= E_o^\omega[\exp\{\langle\theta,X_{\tau_N}\rangle-\Lambda_a(\theta)\tau_N\}]\qquad\mbox{and}\\
W_N(\theta,\omega)&:= E_o^\omega[\left.\exp\{\langle\theta,X_{\tau_N}\rangle-\Lambda_a(\theta)\tau_N\}\right|\beta=\infty].
\end{align*}

\begin{lemma}\label{cevat}
For every $\theta\in\mathcal{C}(c_3)$, if 
$$\limsup_{N\to\infty}\frac{1}{N}\log\hat{W}_N(\theta,\cdot)<0$$ holds $\mathbb{P}$-a.s., then $\Lambda_q(\theta)<\Lambda_a(\theta)$.
\end{lemma}

\begin{proof}
Let $\theta\in\mathcal{C}(c_3)$. Then, $\theta\in\mathcal{C}(c)$ for some $c<c_3$. By hypothesis, for $\mathbb{P}$-a.e.\ $\omega$, there exist $C_3\geq1$ and $c_4>0$ (both depending on $\omega$) such that $\hat{W}_N(\theta,\omega)\leq C_3\mathrm{e}^{-c_4N}$ for every $N\geq1$.

Given any $n\geq1$ and $K\geq1$, it follows from Chebyshev's inequality and (\ref{yoruk}) that
\begin{align*}
&E_o^\omega[\exp\{\langle\theta,X_n\rangle-\Lambda_a(\theta)n\}]\\
&\qquad=E_o^\omega[\exp\{\langle\theta,X_n\rangle-\Lambda_a(\theta)n\},n<\tau_{\lfloor\frac{n}{K}\rfloor}]+\sum_{j=\lfloor\frac{n}{K}\rfloor}^nE_o^\omega[\exp\{\langle\theta,X_n\rangle-\Lambda_a(\theta)n\},\tau_j\leq n<\tau_{j+1}]\\
&\qquad\leq\mathrm{e}^{2|\theta|n}P_o^\omega(n<\tau_{\lfloor\frac{n}{K}\rfloor})+\sum_{j=\lfloor\frac{n}{K}\rfloor}^nE_o^\omega[\exp\{\langle\theta,X_{\tau_j}\rangle-\Lambda_a(\theta)\tau_j\}]\mathrm{ess}\sup_{\mathbb{P}}E_o^{\omega'}[\left.\exp\{2|\theta|\tau_1\}\right|\beta=\infty]\\
&\qquad\leq\mathrm{e}^{(2|\theta|-c)n}E_o^{\omega}[\exp\{c\tau_{\lfloor\frac{n}{K}\rfloor}\}]+\sum_{j=\lfloor\frac{n}{K}\rfloor}^n\hat{W}_j(\theta,\omega)\mathrm{ess}\sup_{\mathbb{P}}E_o^{\omega'}[\left.\exp\{c\tau_1\}\right|\beta=\infty]\\
&\qquad\leq\mathrm{e}^{(2|\theta|-c)n}E_o^{\omega}[\exp\{c\tau_1\}]\left(\mathrm{ess}\sup_{\mathbb{P}}E_o^{\omega'}[\left.\exp\{c\tau_1\}\right|\beta=\infty]\right)^{\lfloor\frac{n}{K}\rfloor-1}\\
&\qquad\quad+\sum_{j=\lfloor\frac{n}{K}\rfloor}^n\hat{W}_j(\theta,\omega)\mathrm{ess}\sup_{\mathbb{P}}E_o^{\omega'}[\left.\exp\{c\tau_1\}\right|\beta=\infty]\\
&\qquad\leq\mathrm{e}^{(2|\theta|-c)n}H(c)^{\lfloor\frac{n}{K}\rfloor}+H(c)\sum_{j=\lfloor\frac{n}{K}\rfloor}^nC_3\mathrm{e}^{-c_4j}.
\end{align*}
Take $K$ sufficiently large, and conclude that $$\Lambda_q(\theta)-\Lambda_a(\theta)=\lim_{n\to\infty}\frac{1}{n}\log E_o^\omega[\exp\{\langle\theta,X_n\rangle-\Lambda_a(\theta)n\}]<0.\qedhere$$
\end{proof}

\begin{lemma}\label{bomontiymis}
For every $\theta\in\mathcal{C}(c_3)$, if 
\begin{equation}\label{serce}
\limsup_{N\to\infty}
\frac{1}{N}\log\mathbb{E}\left[ W_N(\theta,\cdot)^\alpha\right]<0
\end{equation} for some $\alpha\in(0,1)$, then $\Lambda_q(\theta)<\Lambda_a(\theta)$. Hence, by convex duality, $I_a<I_q$ at $\xi=\nabla\Lambda_a(\theta)$.
\end{lemma}

\begin{proof}
For any $N\geq1$ and $\theta\in\mathcal{C}(c_3)$, it follows from the renewal structure and (\ref{cevikuff}) that
\begin{align*}
\mathbb{E}[P_o^\omega(\beta=\infty)W_N(\theta,\cdot)]&=P_o(\beta=\infty)E_o[\left.\exp\{\langle\theta,X_{\tau_N}\rangle-\Lambda_a(\theta)\tau_N\}\right|\beta=\infty]\\
&=P_o(\beta=\infty)\left(E_o[\left.\exp\{\langle\theta,X_{\tau_1}\rangle-\Lambda_a(\theta)\tau_1\}\right|\beta=\infty]\right)^N\\
&=P_o(\beta=\infty).
\end{align*}
Given any $\alpha\in(0,1)$, by the same reasoning as in (\ref{fracmom}),
\begin{equation}\label{hakki}
\limsup_{N\to\infty}\frac{1}{N}\log\hat{W}_N(\theta,\cdot)\leq
\limsup_{N\to\infty}\frac{1}{N\alpha}
\log\mathbb{E}\left[\hat{W}_N(\theta,\cdot)^\alpha\right],\quad
\mathbb{P}\mbox{-a.s.}
\end{equation}
On the other hand, if $2|\theta|<c<c_3$, then we see by subadditivity, Chebyshev's inequality, and (\ref{yoruk}) that
\begin{align}
\mathbb{E}\left[\hat{W}_{N+1}(\theta,\cdot)^\alpha\right]&=\mathbb{E}\left[\left(\sum_{x\in\mathbb{Z}^d}E_o^\omega[\exp\{\langle\theta,X_{\tau_1}\rangle-\Lambda_a(\theta)\tau_1\},X_{\tau_1}=x]W_N(\theta,T_x\cdot)\right)^\alpha\right]\nonumber\\
&\leq\mathbb{E}\left[\sum_{x\in\mathbb{Z}^d}\left(E_o^\omega[\exp\{\langle\theta,X_{\tau_1}\rangle-\Lambda_a(\theta)\tau_1\},X_{\tau_1}=x]\right)^\alpha W_N(\theta,T_x\cdot)^\alpha\right]\nonumber\\
&\leq\mathbb{E}\left[\sum_{x\in\mathbb{Z}^d}\left(E_o^\omega[\exp\{2|\theta|\tau_1\},\tau_1\geq|x|_1]\right)^\alpha W_N(\theta,T_x\cdot)^\alpha\right]\nonumber\\
&\leq\mathbb{E}\left[\sum_{x\in\mathbb{Z}^d}\left(\mathrm{e}^{(2|\theta|-c)|x|_1}E_o^\omega[\exp\{c\tau_1\}]\right)^\alpha W_N(\theta,T_x\cdot)^\alpha\right]\nonumber\\
&\leq H(c)^\alpha\mathbb{E}\left[ W_N(\theta,\cdot)^\alpha\right]\sum_{x\in\mathbb{Z}^d}\mathrm{e}^{(2|\theta|-c)\alpha|x|_1}.\label{dogar}
\end{align}
The desired result follows immediately from (\ref{hakki}), (\ref{dogar}) and Lemma \ref{cevat}.
\end{proof}

\subsection{The correlation condition}

In this subsection, we will consider space-only RWRE on $\mathbb{Z}^d$ with $d=2,3$, assume that the walk is non-nestling relative to $e_d$, and outline how one can modify the arguments given in Section \ref{SpaceTimeSection} in order to reduce (\ref{serce}) to a simpler inequality.

We start with $d=2$. For every $n\geq1$ of the form $k^2$, and for every $y=(y',y'')\in\mathbb{Z}^2$, let $$J_y:=[(y'-\frac{1}{2})\sqrt{n},(y'+\frac{1}{2})\sqrt{n})\times[(y''-\frac{1}{2})\sqrt{n},(y''+\frac{1}{2})\sqrt{n})\subset\mathbb{R}^2,$$c.f.\ (\ref{azkalbir}). Take $N=nm$ for some $m\geq1$. For every $\theta\in\mathcal{C}(c_3)$, $\omega\in\Omega$ and $Y=(y_1,\ldots,y_m)\in(\mathbb{Z}^2)^m$, define $$\bar{W}_N(\theta,\omega,Y):=E_o^\omega[\left.\exp\{\langle\theta,X_{\tau_N}\rangle-\Lambda_a(\theta)\tau_N\}, X_{\tau_{jn}}-\lfloor jn\zeta(\theta)\rfloor\in J_{y_j}\mbox{ for every }j\leq m\right|\beta=\infty],$$c.f.\ (\ref{azkaliki}), 
where 
%OO1
\begin{equation}
\label{eq-150310}
\zeta(\theta):=E_o[\left.X_{\tau_1}\exp\{\langle\theta,X_{\tau_1}\rangle-
\Lambda_a(\theta)\tau_1\}\right|\beta=\infty].
\end{equation}
By subadditivity,$$\mathbb{E}[W_N(\theta,\cdot)^\alpha]\leq\sum_{Y}\mathbb{E}\left[\bar{W}_N(\theta,\cdot,Y)^\alpha\right],$$ c.f.\ (\ref{asama}). Given any $C_1\geq1$, $Y=(y_1,\ldots,y_m)\in(\mathbb{Z}^2)^m$ and $j\in\{1,\ldots,m\}$, let
\begin{align*}
B_j=B_j(y_{j-1},y_j)&:=\{(s,i)\in\mathbb{Z}^2: (j-1)n\langle\zeta(\theta),e_2\rangle+\sqrt{n}(y_{j-1}''+1/2)\leq i<jn\langle\zeta(\theta),e_2\rangle+\sqrt{n}(y_j''-1/2),\\
&\hspace{5.8cm}|(s-\sqrt{n}y_{j-1}')-\frac{\langle\zeta(\theta),e_1\rangle}{\langle\zeta(\theta),e_2\rangle}(i-\sqrt{n}y_{j-1}'')|\leq C_1\sqrt{n}\},
\end{align*}
c.f.\ (\ref{bicey}). Also, redefine $a(\theta,\cdot)$ by setting $$a(\theta,x):=\langle\theta,v(T_x\omega)\rangle - \mathbb{E}[\langle\theta,v(\cdot)\rangle]$$
for every $x\in\mathbb{Z}^2$, where $v(\omega)=\sum_{z\in\mathcal{R}}\pi(0,z)z$ as before. Note that, under the assumptions stated in Definition \ref{gaffur}, we have $\mathbb{E}[\langle\theta,v(\cdot)\rangle]=\langle\theta,\xi_o\rangle$. However, this equality does not necessarily hold in general.

With these modified definitions, the arguments in Subsections 
\ref{tiltsubsec} and \ref{estsubsec} easily carry over, once one replaces the i.i.d.\ random variables $$E_{o}^{T_{X_i}\omega}[\exp\{\langle \theta,X_1\rangle-\log \phi(\theta)\}]$$
by the variables
%OO1 misplaced bracket
$$E_o^{T_{X_{\tau_i}}\omega}[\exp\{\langle\theta,X_{\tau_1}\rangle-
\Lambda_a(\theta)\tau_1\}|\beta=\infty]\,.$$
Therefore, in order to prove (\ref{serce}), it suffices to show that
\begin{equation}\label{lemansamndd}
\sum_{y\in\mathbb{Z}^2}\max_{x\in J_o}E_x\left[\exp\{\langle\theta,X_{\tau_n}-x\rangle-\Lambda_a(\theta)\tau_n+f_K(\delta_nD(B_1))\},X_{\tau_n}-\lfloor n\zeta(\theta)\rfloor\in J_y\,|\,\beta=\infty\right]^\alpha<1/2
\end{equation}
when 
%OO1
$\zeta(\theta)$ is as in (\ref{eq-150310}) and 
$n$, $K$, $C_1$ are sufficiently large, c.f.\ Lemma \ref{lessthanhalf}. Here, $\alpha\in(0,1)$ is fixed, $f_K(u):=-K\one_{u\geq\mathrm{e}^{K^2}}$ and $\delta_n=C_1^{-1/2}n^{-3/4}$, as before.

We imitate (\ref{twosums}), and write the sum in (\ref{lemansamndd}) as
\begin{equation}\label{twosumsndd}
\sum_{y\in\mathbb{Z}^2}\max_{x\in J_o}E_x\left[\cdots\right]^\alpha=\sum_{{y\in\mathbb{Z}^2:}\atop{|y|>R}}\max_{x\in J_o}E_x\left[\cdots\right]^\alpha+\sum_{{y\in\mathbb{Z}^2:}\atop{|y|\leq R}}\max_{x\in J_o}E_x\left[\cdots\right]^\alpha
\end{equation} with some large constant $R$,
to be determined. Just like in the space-time case, the first sum on the RHS of (\ref{twosumsndd}) is bounded from above by
\begin{equation}\label{okcc}
\sum_{{y\in\mathbb{Z}^2:}\atop{|y|>R}}\hat{P}_o^\theta\left(\left|\frac{X_{n}-\lfloor n\zeta(\theta)\rfloor}{\sqrt{n}}\right|\geq |y|-1\right)^\alpha.
\end{equation}
Here, $\hat P_o^\theta$ is redefined to be the probability measure on paths induced by the random walk (in a deterministic environment) whose transition probabilities are given by
$$q^\theta(x):=E_o[\left.\exp\{\langle\theta,X_{\tau_1}\rangle-\Lambda_a(\theta)\tau_1\}, X_{\tau_1}=x\right|\beta=\infty],\quad x\in\mathbb{Z}^2.$$
(Note that $\sum_{x\in\mathbb{Z}^2}q^\theta(x)=1$ by (\ref{cevikuff}).) If $\hat E_o^\theta$ denotes the corresponding expectation, it is clear that \begin{equation}\label{ozggurk}
\hat E_o^\theta[\exp\{c|X_1|\}]<\infty\quad\mbox{ for every }c\in(0,c_3-2|\theta|).
\end{equation}
Therefore, by Chebyshev's inequality, (\ref{okcc}) can be made arbitrarily small (uniformly in large $n$) by choosing $R$ sufficiently large.

The second sum on the RHS of (\ref{twosumsndd}) can be controlled by showing that 
\begin{equation}\label{oreoye}
\max_{{y\in\mathbb{Z}^2:}\atop{|y|\leq R}}\max_{x\in J_o}E_x\left[\exp\{\langle\theta,X_{\tau_n}-x\rangle-\Lambda_a(\theta)\tau_n+f_K(\delta_nD(B_1))\},X_{\tau_n}-\lfloor n\zeta(\theta)\rfloor\in J_y\,|\,\beta=\infty\right]
\end{equation} is small when $n$, $K$ and $C_1$ are sufficiently large. In the space-time case, the verification of the analogous statement, i.e., (\ref{lastsuff}), relied on the fact that \begin{equation}\label{krusil}
E_o[\exp\{\langle\theta,X_{n}\rangle-
n\log\phi(\theta)\}\sum_{i=0}^{n-1}a(\theta,X_i)]
\end{equation}
grows linearly in $n$, c.f.\ (\ref{chebyineq}) and 
\eqref{crossterms}. 
In the space-only case, the drift vectors at the points off the path do not contribute to the mean of $D(B_1)$ under the tilted measure, and the drift vector at any point on the path contributes only once even if it is visited multiple times. Therefore, 
%OO1
the statement concerning
(\ref{krusil}) needs to be replaced by the statement that
\begin{equation}\label{krusilyeni}
E_o[\exp\{\langle\theta,X_{\tau_n}\rangle-\Lambda_a(\theta)\tau_n\}\!\!\!\!\!\sum_{x\in S(X,\tau_n)}\!\!\!\!\!a(\theta,x)|\beta=\infty]
\end{equation} grows linearly in $n$.
Here, for any $j\geq1$,
\begin{equation}\label{isimmm}
S(X,j):=\{X_i:0\leq i<j\}.
\end{equation}

In the space-time case, the variance of $D(B_1)$ under the tilted measure was 
%OO1 why annoy the ref?  easily 
shown to be $O(n^{3/2})$ since the only non-vanishing terms were those corresponding to points $x,y\in\mathbb{Z}^2$ such that $x=y$. In the space-only case, steps of the walk between consecutive regeneration times are not independent, and we therefore need to also consider terms corresponding to $x$ and $y$ that are both on the path in the same regeneration block.  However, since regeneration times have exponentially decaying tails, the total contribution of such terms is $O(n)$, and the variance of $D(B_1)$ under the tilted measure is still $O(n^{3/2})$.

With these modifications, the argument in Subsection \ref{finsubsec} enables us to deduce (\ref{serce}) provided that (\ref{krusilyeni}) grows linearly in $n$. By the renewal structure, the latter is equivalent to the following \textit{correlation condition}:
%OO1
\begin{equation}\label{theassumption}
\mu:=E_o[\exp\{\langle\theta,X_{\tau_1}\rangle-\Lambda_a(\theta)\tau_1\}\!\!\!\!\!\sum_{x\in S(X,\tau_1)}\!\!\!\!\!a(\theta,x)|\beta=\infty]>0.
\end{equation} 
(This replaces the choice of $\mu$ for the space-time case, see 
(\ref{eq-150310a}).)

For $d=3$, after modifying (\ref{oreoye}) by (i) taking the first maximum over $\{y\in\mathbb{Z}^3: |y|\leq R\}$, (ii) replacing the sets $J_y$ and $B_1$ by their three dimensional analogs, and (iii) redefining $D(B_1)$ as in (\ref{yeniD}), one can employ the reasoning above in order to reduce (\ref{serce}) to showing that (\ref{oreoye}) is small when $n$, $K$ and $C_1$ are sufficiently large. After that, one can set $\delta_n:=n^{-1}(\log n)^{-1/2}$, apply the same kind of modifications to the argument given in Subsection \ref{ayyy}, and further reduce (\ref{serce}) to (\ref{theassumption}). In particular, note that Lemma \ref{localCLTrefined} continues to hold under the new definition of $\hat P_o^\theta$, thanks to (\ref{ozggurk}). We omit the (routine) details.

We have arrived at the following theorem.
\begin{theorem}\label{nesin}
%OO1
Consider space-only
RWRE on $\mathbb{Z}^d$ with $d=2,3$. Assume (\ref{cakmaeminabi}), (\ref{space-only}) and that the walk is non-nestling relative to $e_d$. 
Then, there exists an open set $\mathcal{A}_{so}\subset\mathbb{R}^d$ with the following properties:
\begin{itemize}
\item[(i)] $I_a$ is strictly convex and analytic on $\mathcal{A}_{so}$,
\item[(ii)] $\xi_o\in\mathcal{A}_{so}$, and 
\item[(iii)] for every $\xi\in\mathcal{A}_{so}$, the strict inequality $I_a(\xi)<I_q(\xi)$ holds if (\ref{theassumption}) is satisfied at $\theta:=\nabla I_a(\xi)$.
\end{itemize}
\end{theorem}

\begin{proof}
Recall (\ref{metkelsi}), and define $$\mathcal{A}_{so}:=\{\nabla\Lambda_a(\theta):\theta\in\mathcal{C}(c_3)\}.$$ It follows from Lemma \ref{zikkim} and the inverse function theorem that $I_a$ is strictly convex and analytic on $\mathcal{A}_{so}$ which is an open set containing $\xi_o$.

Take any $\xi\in\mathcal{A}_{so}$. Note that $\theta:=\nabla I_a(\xi)$ satisfies $\xi=\nabla\Lambda_a(\theta)$ by convex duality. As outlined above, (\ref{theassumption}) implies (\ref{serce}). Hence, the desired result follows from Lemma \ref{bomontiymis}.
\end{proof}

\subsection{Proof of Theorem \ref{QneqAnn}}

Consider space-only RWRE on $\mathbb{Z}^d$ with $d=2,3$. Fix a triple $p=(p^+,p^o,p^-)$ of positive real numbers such that $p^-<p^+$ and $p^++p^o+p^-=1$. Assume that $\mathbb{P}$ is in class $\mathcal{M}_\epsilon(d,p)$ for some small $\epsilon>0$, c.f.\ Definition \ref{gaffur}. Assume that $\epsilon\leq\frac{p^o}{4(d-1)}$ so that the ellipticity constant $\kappa$ of the walk satisfies 
\begin{equation}\label{kapis}
\kappa\geq\min\left(p^+,p^-,\frac{p^o}{4(d-1)}\right).
\end{equation}
\begin{lemma}\label{alisan}
There exist $C_4\geq1$ and $c_5>0$ (depending only on $p$) such that $|\Lambda_a(\theta)-\langle\theta,\xi_o\rangle|\leq C_4|\theta|^2$ holds for every $\theta\in\mathcal{C}(c_5)$.
\end{lemma}
\begin{proof}
Recall (\ref{serdar}). Note that $c_3$ depends only on the law of the regeneration times which, in turn, is determined by the fixed triple $p$. Moreover, the ellipticity constant $\kappa$ of the walk satisfies (\ref{kapis}). Fix any $c_5<c_3$. The desired result follows immediately from Lemma \ref{zikkim}.
\end{proof}

Consider the set $$\mathcal{C}_t(c_5):=\{\theta\in\mathcal{C}(c_5):\langle\theta,e_d\rangle=0\}.$$
(Here, the subscript stands for \textit{transversal}.) Take any $\theta\in\mathcal{C}_t(c_5)$. Recall the notation in (\ref{isimmm}). Since $\mathbb{P}$ is in class $\mathcal{M}_\epsilon(d,p)$, it is easy to see that
\begin{equation}\label{cokazkaldihaci}
\xi_o=(p^+-p^-)e_d,\qquad\langle\theta,\xi_o\rangle=0,\qquad\mbox{and}\qquad |a(\theta,x)|=|\langle\theta,v(T_x\omega)\rangle|\leq2\epsilon(d-1)|\theta|
\end{equation}
for every $x\in\mathbb{Z}^d$. Similarly, the isotropy assumption ensures that
$$Z(\theta)=Z(\theta,X,\tau_1,\omega):=\sum_{x\in S(X,\tau_1)}\!\!\!\!\!a(\theta,x)$$
satisfies
\begin{equation}\label{isara}
E_o[Z(\theta)|\beta=\infty]=E_o[\tau_1 Z(\theta)|\beta=\infty]=0.
\end{equation}

Our aim is to show that
\begin{equation}\label{cc}
E_o[\exp\{\langle\theta,X_{\tau_1}\rangle-\Lambda_a(\theta)\tau_1\}Z(\theta)|\beta=\infty]>0
\end{equation} 
for certain choices of $\theta$, to be determined later. Expanding the exponential on the LHS of (\ref{cc}), we see that
\begin{align}
E_o[\exp\{\langle\theta,X_{\tau_1}\rangle-\Lambda_a(\theta)\tau_1\}Z(\theta)|\beta=\infty]&\geq E_o[(1+\langle\theta,X_{\tau_1}\rangle-\Lambda_a(\theta)\tau_1)Z(\theta)|\beta=\infty]-C_5|\theta|^3\label{gurrbb}\\
&= E_o[\langle\theta,X_{\tau_1}\rangle Z(\theta)|\beta=\infty]-C_5|\theta|^3.\label{kuntin}
\end{align}
%OO1 Explanation: 
Indeed,
(\ref{gurrbb}) follows from $|Z(\theta)|\leq 2\epsilon(d-1)|\theta|\tau_1$ 
and
$$1+\langle\theta,X_{\tau_1}\rangle-\Lambda_a(\theta)\tau_1\leq
\exp\{\langle\theta,X_{\tau_1}\rangle-\Lambda_a(\theta)\tau_1\}\leq1+
\langle\theta,X_{\tau_1}\rangle-\Lambda_a(\theta)\tau_1 + 
2|\theta|^2\tau_1^2\exp\{2|\theta|\tau_1\},$$ 
$C_5$ is some constant that depends only on $p$ and $c_5$ and
finally, (\ref{isara}) implies \eqref{kuntin}.
 
In order to estimate (\ref{kuntin}), we first provide a more convenient 
%OO1
representation of the
RWRE. Let $(b_i)_{i\geq0}$ be an i.i.d.\ sequence of random variables taking values in $\{e_d,0,-e_d\}$, with 
$$P(b_1=e_d)=p^+,\quad P(b_1=0)=p^o,\quad\mbox{and}\quad P(b_1=-e_d)=p^-.$$
Let $(f_i)_{i\geq0}$ be another i.i.d.\ sequence of random variables (independent of $(b_i)_{i\geq0}$) taking values in the set $\{\pm e_j:1\leq j<d\}\cup\{0\}$, with $$P(f_1=0)=\frac{2\epsilon(d-1)}{p^o}\quad\mbox{and}\quad P(f_1=\pm e_j)=\frac{1}{2(d-1)}-\frac{\epsilon}{p^o}\quad\mbox{if}\ 1\leq j<d.$$
For any $\omega\in\Omega$, the walk $(X_i)_{i\geq0}$ under $P_o^\omega$ can be constructed by setting
$$X_{i+1}-X_i:=b_i+(1-|b_i|)f_i+(1-|b_i|)(1-|f_i|)U_i\,,$$
where $(U_i)_{i\geq0}$ is a sequence of independent random variables taking values in $\{\pm e_j:1\leq j<d\}$, with
$$P^\omega(U_i=\pm e_j|\mathcal{F}_i)=\frac{\pi(X_i,X_i\pm e_j)-(\frac{p^o}{2(d-1)}-\eps)}{2\epsilon(d-1)}.$$
Here, ${\mathcal F}_i=\sigma(X_1,\ldots,X_i)$. Note that the laws of the sequences $(b_i)_{i\geq0}$ and $(f_i)_{i\geq0}$ do not depend on the environment, 
and that $\tau_1$ is a function of $(b_i)_{i\geq0}$ only.

Let 
$$N_i:=\sum_{j=0}^{i-1}\one_{1=(1-|b_i|)(1-|f_i|)}\,.$$
Introduce the events
$L_0:=\{ N_{\tau_1}=0\}$, $L_1:=\{ N_{\tau_1}=1\}$, and $L_2:=\{ N_{\tau_1}\geq 2\}$.
Let $\mathcal{G}:=\sigma((b_i,f_i)_{i\geq0})$.
Note that the events $L_0,L_1$ and $L_2$ are ${\mathcal G}$-measurable, and so is
the event $\{\beta=\infty\}$.
On the event $L_0$, the walker never sees the environment until $\tau_1$, and thus $X_{\tau_1}$ is ${\mathcal G}$-measurable.
Also, for any $i\geq0$, on the event $\{X_i\notin S(X,i)\}$ (i.e., when $X_i$ is a fresh point), $a(\theta,X_i)$ is independent of $\mathcal{F}_i$ and $\mathcal{G}$ under $P_o$.
Therefore, by isotropy,
$$E_o[Z(\theta)|\mathcal{G}]=E_o\left[\left.\sum_{i=0}^{\tau_1-1}a(\theta,X_i)\one_{X_i\notin S(X,i)}\right|\mathcal{G}\right]=0.$$ 
Putting these observations together, we see that
\begin{equation}\label{birincicinko}
E_o[\langle\theta,X_{\tau_1}\rangle Z(\theta),L_0,\beta=\infty]=E_o[\langle\theta,X_{\tau_1}\rangle E_o[Z(\theta)|\mathcal{G}],L_0,\beta=\infty]=0.
\end{equation}

On the other hand, it is easy to check that 
$P_o(L_2)\leq c_6\epsilon^2$ for some $c_6=c_6(p)$. By H\"{o}lder's inequality,
\begin{equation}\label{ikincicinko}
|E_o[\langle\theta,X_{\tau_1}\rangle Z(\theta),L_2,\beta=\infty]|\leq P_o(L_2)^{2/3}E_o[|\langle\theta,X_{\tau_1}\rangle Z(\theta)|^3,\beta=\infty]^{1/3}\leq c_7\epsilon^{7/3}|\theta|^2
\end{equation}
for some $c_7=c_7(p)>0$. (Recall that $a(\theta,\cdot)\leq2\epsilon(d-1)|\theta|$, c.f.\ (\ref{cokazkaldihaci}).)

Finally, let 
$L_1^\ell=L_1\cap\{(1-|b_\ell|)(1-|f_\ell|)=1, \ell<\tau_1\}$. Then,
\begin{equation}\label{tombala}
E_o[\langle\theta,X_{\tau_1}\rangle Z(\theta),L_1,\beta=\infty]=\sum_{\ell=0}^\infty E_o[\langle\theta,X_{\tau_1}\rangle Z(\theta),L_1^\ell,\beta=\infty].
\end{equation}
For every $\ell\geq0$,
\begin{align}
E_o[\langle\theta,X_{\tau_1}\rangle Z(\theta),L_1^\ell,\beta=\infty]&=E_o[\langle\theta,X_\ell\rangle Z(\theta),L_1^\ell,\beta=\infty]\label{zerrin}\\
&\quad+E_o[\langle\theta,X_{\ell+1}-X_\ell\rangle Z(\theta),L_1^\ell,\beta=\infty]\nonumber\\
&\quad+E_o[\langle\theta,X_{\tau_1}-X_{\ell+1}\rangle Z(\theta),L_1^\ell,\beta=\infty].\nonumber
\end{align}
By computations similar to the one involving $L_0$, the first and the third terms on the RHS of (\ref{zerrin}) are zero.
The second term is equal to
\begin{align*}
&E_o[\langle\theta,X_{\ell+1}-X_\ell\rangle\langle\theta,v(T_{X_\ell}\omega)\rangle,L_1^\ell,\beta=\infty]=P_o(L_1^\ell,\beta=\infty)\mathbb{E}[E^\omega[\langle\theta,U_o\rangle]\langle\theta,v(\omega)\rangle]\\
&\qquad=P_o(L_1^\ell,\beta=\infty)\mathbb{E}\left[\left(\sum_{z\neq\pm e_d}\frac{\pi(0,z)-(\frac{p^o}{2(d-1)}-\eps)}{2\epsilon(d-1)}\langle\theta,z\rangle\right)\langle\theta,v(\omega)\rangle\right]\\
&\qquad=\frac{P_o(L_1^\ell,\beta=\infty)}{2\epsilon(d-1)}\mathbb{E}\left[\langle\theta,v(\omega)\rangle^2\right].
\end{align*}
Therefore, by (\ref{tombala}),
$$E_o[\langle\theta,X_{\tau_1}\rangle Z(\theta),L_1,\beta=\infty]=\frac{P_o(L_1,\beta=\infty)}{2\epsilon(d-1)}\mathbb{E}\left[\langle\theta,v(\omega)\rangle^2\right].$$
It is easy to see that $P_o(L_1,\beta=\infty)\geq c_8\epsilon$ 
for some $c_8=c_8(p)>0$ if $\epsilon$ is small enough. Also, part (c) of Definition \ref{gaffur} ensures that $\mathbb{E}\left[\langle\theta,v(\omega)\rangle^2\right]\geq c_9\epsilon^2|\theta|^2$ for some $c_9=c_9(p)>0$. Hence,
\begin{equation}\label{cardak}
E_o[\langle\theta,X_{\tau_1}\rangle Z(\theta),L_1,\beta=\infty]\geq c_{10}\epsilon^2|\theta|^2
\end{equation}  for some $c_{10}=c_{10}(p)>0$.
Combining (\ref{birincicinko}), (\ref{ikincicinko}) and (\ref{cardak}) gives
\begin{align}
E_o[\langle\theta,X_{\tau_1}\rangle Z(\theta)|\beta=\infty]-C_5|\theta|^3&\geq c_{10}\epsilon^2|\theta|^2 - c_7\epsilon^{7/3}|\theta|^2 - C_5|\theta|^3\label{obarey}\\
&=\left((c_{10} - c_7\epsilon^{1/3})\epsilon^2 - C_5|\theta|\right)|\theta|^2.\nonumber
\end{align}
If $\epsilon<(c_{10}/c_7)^3$, then, for every $\theta\in\mathcal{C}_t(c_5)$ such that $0<|\theta|<(c_{10} - c_7\epsilon^{1/3})\epsilon^2/C_5$, 
$$E_o[\exp\{\langle\theta,X_{\tau_1}\rangle-\Lambda_a(\theta)\tau_1\}Z(\theta)|\beta=\infty]>0$$
by (\ref{kuntin}) and (\ref{obarey}).

Finally, Theorem \ref{nesin} implies that $I_a<I_q$ on the set $$\{\nabla\Lambda_a(\theta): \theta\in\mathcal{C}_t(c_5),\ 0<|\theta|<(c_{10} - c_7\epsilon^{1/3})\epsilon^2/C_5\}$$ whose closure contains the LLN velocity $\xi_o=\nabla\Lambda_a(0)$. We have proved Theorem \ref{QneqAnn}.

\section{Open problems}
\label{sec-openprob}
Our technique of proof puts several restrictions 
on  the class of models treated. The following are natural questions we have 
not addressed.
\begin{enumerate}
\item Does Theorem \ref{QneqAnn} extend to all space-only
RWRE in dimension  $d=2,3$, or at least to those satisfying
Sznitman's condition ({\bf{T}})? Note that, for non-nestling walks, it suffices to show that the correlation condition (\ref{theassumption}) is satisfied on a sequence $(\theta_n)_{n\geq1}$ that converges to zero, c.f.\ Theorem \ref{nesin}.
\item In case $\sum \pi(0,z)\langle z,e\rangle$ is random for any
$e\in \mathcal{R}_{so}$,
is it true that $I_q(\xi)=I_a(\xi)$ only 
when $\xi=0$ or $I_a(\xi)=0$, as is the case in dimension $d=1$?
\end{enumerate}

%OO1
In our proof of Theorem \ref{QneqAnn} (specifically,
in the proof of the correlation condition
(\ref{theassumption})),
we used the isotropy 
assumption in order to get rid of a centering term under the
(untilted) measure; this does not seem essential and probably, the lack of 
isotropy could be handled in the perturbative regime. 
However, getting rid of
the perturbative restriction, 
or of the non-randomness in the $e_d$ direction, 
requires additional arguments. 
\section*{Acknowledgments}
This research was supported partially by a grant from the 
Israeli Science Foundation, and by the 
Alhadeff Fund at the Weizmann Institute. 
We thank Francis Comets for providing us with an update on 
polymer models and bringing the work of Lacoin \cite{Lacoin}
to our attention.

\bibliographystyle{plain}
\bibliography{atilla_references}
\end{document}